\documentclass{amsart}
\usepackage{amsthm}
\usepackage{amssymb}
\usepackage{tikz-cd}
\tikzset{labl/.style={anchor=south, rotate=270, inner sep=.5mm}}

\usepackage{colonequals}
\usepackage{dsfont}
\usepackage{enumerate}
\usepackage{eucal}
\usepackage{color, comment}

\usepackage{mathtools}
\usepackage{todonotes}

\usepackage{hyperref}
\hypersetup{%
  bookmarksnumbered=true,%
  colorlinks=true,%
  linkcolor=blue,%
  citecolor=blue,%
  filecolor=blue,%
  menucolor=blue,%
  urlcolor=blue,%
  bookmarksopen=true,%
  bookmarksdepth=2,%
  pageanchor=true}

\makeatletter
\@namedef{subjclassname@2020}{\textup{2020} Mathematics Subject Classification}
\makeatother

\numberwithin{equation}{section}

\theoremstyle{plain}

\newcounter{intro}

\newtheorem{introthm}[intro]{Theorem}

\newtheorem{theorem}[equation]{Theorem}
\newtheorem{proposition}[equation]{Proposition}
\newtheorem{lemma}[equation]{Lemma} 
\newtheorem{corollary}[equation]{Corollary}

\theoremstyle{definition}

\newtheorem{example}[equation]{Example}

\newtheorem{chunk}[equation]{}

\theoremstyle{remark}

\newtheorem*{ack}{Acknowledgements}

\newcommand{\add}{\operatorname{add}}

\newcommand{\bfD}{\mathbf{D}}
\newcommand{\bfK}{\mathbf{K}}
\newcommand{\bfp}{\mathbf{p}}
\newcommand{\bfi}{\mathbf{i}}

\newcommand{\bfq}{\mathbf{q}}
\newcommand{\bft}{\mathbf{t}}

\newcommand{\bbF}{\mathbb F} 
\newcommand{\bbQ}{\mathbb Q} 
\newcommand{\bbZ}{\mathbb Z}

\newcommand{\cat}[1]{\mathcal{#1}}
\newcommand{\Ch}{\operatorname{Ch}} 

\newcommand{\Coker}{\operatorname{Coker}}
\newcommand{\colim}{\operatorname*{colim}}
\newcommand{\chu}{U}
\newcommand{\crep}{\operatorname{\mathcal{R}\!\!\;\mathit{ep}}}

\newcommand{\dcat}[1]{\mathbf{D}(\operatorname{Mod} #1)}

\newcommand{\enhK}{\mathcal{K}} 
\newcommand{\enhD}{\mathcal{D}} 
 
\newcommand{\End}{\operatorname{End}}
\newcommand{\eps}{\varepsilon}
\newcommand{\Ext}{\operatorname{Ext}}

\newcommand{\fibre}[2]{ {#1}_{k(#2)}}
\newcommand{\lfMod}{\operatorname{Mod}^{\mathrm{lf}}}
\newcommand{\Fun}{\operatorname{Fun}}

\newcommand{\fm}{\mathfrak{m}} 
\newcommand{\fp}{\mathfrak{p}}
\newcommand{\fq}{\mathfrak{q}}

\newcommand{\Hom}{\operatorname{Hom}}
\newcommand{\fHom}{\operatorname{\mathcal{H}\!\!\;\mathit{om}}}

\newcommand{\Ind}{\operatorname{Ind}} 
\newcommand{\Inj}{\operatorname{Inj}}

\newcommand{\gam}{\varGamma} 

\newcommand{\Ker}{\operatorname{Ker}}
\newcommand{\kos}[2]{{#1}/\!\!/{#2}} 
\newcommand{\lfK}{\mathbf{K}^{\mathrm{lf}}}
\newcommand{\KProj}[1]{\mathbf{K}(\Prj #1)}

\newcommand{\Loc}{\operatorname{Loc}}
\newcommand{\lotimes}{\otimes^{\mathrm L}}

\newcommand{\rmod}{\operatorname{mod}}
\newcommand{\Mod}{\operatorname{Mod}}

\newcommand{\one}{\mathds 1}
\newcommand{\op}{\mathrm{op}}

\newcommand{\Prj}{\operatorname{Proj}}
\newcommand{\prj}{\operatorname{proj}}

\newcommand{\rep}{\operatorname{rep}}

\newcommand{\Rep}{\operatorname{Rep}}

\newcommand{\RHom}{\operatorname{RHom}}

\newcommand{\Spec}{\operatorname{Spec}}
\newcommand{\Spc}{\operatorname{Spc}} 
\newcommand{\St}{\operatorname{St}}
\newcommand{\StlfMod}{\operatorname{StMod}^{\mathrm{lf}}}
\newcommand{\StMod}{\operatorname{StMod}}
\newcommand{\stmod}{\operatorname{stmod}}
\newcommand{\supp}{\operatorname{supp}}
\newcommand{\sslash}{\mathbin{/\mkern-6mu/}}

\newcommand{\Tor}{\operatorname{Tor}}
\newcommand{\Thick}{\operatorname{Thick}}

\newcommand{\lra}{\longrightarrow}
\newcommand{\xra}{\xrightarrow}
\newcommand{\iso}{\xrightarrow{\raisebox{-.4ex}[0ex][0ex]{$\scriptstyle{\sim}$}}}

\newcommand{\longiso}{\xrightarrow{\ \raisebox{-.4ex}[0ex][0ex]{$\scriptstyle{\sim}$}\ }}

\title{Lattices over finite group schemes and stratification}
\author[Barthel, Benson, Iyengar, Krause, and Pevtsova]{Tobias
  Barthel, Dave Benson, Srikanth  B. Iyengar, \\ Henning Krause, and Julia Pevtsova}

\address{Tobias Barthel\\
  Max Planck Institute for Mathematics\\
  Vivatsgasse 7  \\
  53111 Bonn\\
  Germany}

\address{Dave Benson \\ 
Institute of Mathematics\\ 
University of Aberdeen\\ 
King's College\\ 
Aberdeen AB24 3UE\\ 
Scotland U.K.}

\address{Srikanth B. Iyengar\\ 
Department of Mathematics\\
University of Utah\\ 
Salt Lake City, UT 84112\\ 
U.S.A.}

\address{Henning Krause\\ 
Fakult\"at f\"ur Mathematik\\ 
Universit\"at Bielefeld\\ 
33501 Bielefeld\\ 
Germany.}

\address{Julia Pevtsova\\ 
Department of Mathematics\\ 
University of Washington\\ 
Seattle, WA 98195\\ 
U.S.A.}

\begin{document}

\begin{abstract} 
This work concerns representations of a finite flat group scheme $G$, defined over a noetherian commutative ring $R$. The focus is on lattices, namely, finitely generated $G$-modules that are projective as $R$-modules, and on the full subcategory of all $G$-modules projective over $R$ generated by the lattices. The stable category of such $G$-modules is a rigidly-compactly generated, tensor triangulated category. The main result is that this stable category is stratified and costratified by the natural action of the cohomology ring of $G$. Applications include formulas for computing the support and cosupport of tensor products and the module of homomorphisms, and  a classification of the thick ideals in the stable category of lattices.
\end{abstract}



\date{\today}

\maketitle

\setcounter{tocdepth}{1}
\tableofcontents

\section{Introduction}

In recent years, there has been a surge of interest in the structure and properties of tensor triangulated categories arising in algebra, algebraic geometry, and topology. In the algebraic context the starting point is the groundbreaking contributions by Hopkins~\cite{Hopkins:1987a} and Neeman~\cite{Neeman:1992a,Neeman:2011a}. Since then, the subject has expanded into modular representation theory for finite groups, finite group schemes over fields, Lie algebras, rings of differential operators, and equivariant homotopy theory, among others. A central problem is the classification of thick ideals in essentially small tensor triangulated categories, along with the more challenging task of identifying localising and colocalising ideals in the associated big tensor triangulated categories. 

In this work we provide such classification for the big category of representations of finite groups and finite group schemes over commutative noetherian rings. The tensor triangulated category we introduce serves as a unifying framework for two contexts that have been studied extensively.  Specialised to the trivial group it is the derived category of a commutative noetherian ring, and the corresponding classification is due to Neeman. On the other end of the spectrum, when the coefficient ring $R$ is a field, one gets the  stable module category of a finite group or a finite group scheme. Here is our main theorem.

\begin{introthm}
[Theorems ~\ref{th:stratification}, \ref{th:costratification}]
\label{ithm:main}
For $G$ a finite flat group scheme over a commutative noetherian ring $R$, the tensor triangulated category $\Rep(G, R)$ is stratified and costratified by the canonical action of the cohomology ring $H^*(G,R)$. 
\end{introthm}

The category $\Rep(G, R)$ arises as a suitable ind-completion of the bounded derived category $\rep(G, R)$ of \emph{lattices}: representations of $G$ that are finitely generated and projective over $R$. We provide two complementary constructions of this category:  In Sections~\ref{se:Frobenius} and~\ref{se:groups}, we use Frobenius exact categories to give a homological model of $\Rep(G, R)$. In Section~\ref{se:enhancements} we study $\Rep(G, R)$ as the ind-completion in the $\infty$-categorical sense. The ring $H^*(G,R)$, which identifies with $\Ext^*_{RG}(R,R)$, is a graded-commutative noetherian $R$-algebra, by recent work of van der Kallen~\cite{vanderKallen:2023a}.

Group representations over rings are pervasive,  encompassing representations  over rings of integers of algebraic number fields; see for example ~\cite[Chapter~3]{Curtis/Reiner:1981a}. In this context the representations of interest are the torsion-free ones, that is to say, lattices, and these have become objects of study in their own  right~\cite{Reiner:1970a,Swan:1970a,Zimmermann:2022a}. Lattices appear already in the work of Brauer and Nesbitt~\cite{Brauer/Nesbitt:1941a} in the context of integral representations of finite groups  as an intermediary between  ordinary and modular representations.  While classification of indecomposable modules in such categories is usually out of reach, Theorem~\ref{ithm:main} provides structural information on the category $\Rep(G, R)$. Namely, the stratification by $H^*(G,R)$ that it gives   can be restated as a natural bijection between the nonzero localising  ideals of $\Rep(G,R)$, the nonzero colocalising ideals of $\Rep(G,R)$, and subsets of $\Spec H^*(G,R)$:
\[ 
\left\{
  \begin{gathered} \text{localising ideals} \\ \text{of  $\Rep(G,R)$}
\end{gathered}\, \right\} \;\stackrel{\sim}\longleftrightarrow\; 
\left\{\begin{gathered} \text{subsets of}\\ \text{$\Spec H^*(G,R)$} 
\end{gathered}\,  \right\}\,
\;\stackrel{\sim}\longleftrightarrow\; 
\left\{
\begin{gathered} \text{colocalising ideals}\\ \text{of $\Rep(G,R)$} 
\end{gathered}\,  \right \}\,.    
\]

The general results in \cite{Benson/Iyengar/Krause:2011a,Benson/Iyengar/Krause:2012b} allow us to deduce several other desirable consequences from Theorem~\ref{ithm:main}; see Section~\ref{se:costratification} for details.
These include the calculation of the Balmer spectrum \cite{Balmer:2005a} of $\rep(G,R)$:
\begin{equation}
\label{eq:balmer}
\Spc(\rep(G,R)) \longiso \Spec H^*(G,R)\,.
    \end{equation}

This isomorphism encodes a natural bijection between the nonzero thick ideals of $\rep(G,R)$ and specialisation closed subsets of $\Spec H^*(G,R)$.  This is the restriction of the bijection above for the big category $\Rep(G,R)$ to the small category $\rep(G,R)$. In view of the fact that classification of indecomposable lattices is generally a hopeless task, the parametrisation of thick tensor ideals offers a partial solution; it classifies $G$-lattices over $R$ up to homological operations: shifts, extensions, retracts, and tensoring.

\subsection*{Historical remarks and related works} 
As already mentioned, the case $G=1$ in 
Theorem~\ref{ithm:main}  yields the classification for the big 
derived category of $R$-modules and is due to Neeman \cite{Neeman:1992a}. This was preceded by the classification 
of thick subcategories for the bounded derived category of lattices over 
$R$ by Hopkins \cite{Hopkins:1987a}, following the work of  Devinatz, Hopkins, and Smith \cite{Devinatz/Hopkins/Smith:1988a} in stable 
homotopy theory. Studying representations 
of a finite group $G$ over a field $k$,  Benson,  Carlson, and  Rickard \cite{Benson/Carlson/Rickard:1997a} 
classified thick tensor ideals in $\stmod kG$, the small stable category 
of $kG$, which is closely related to $\rep(G,k)$ as we explain in Section~\ref{se:Frobenius}. Continuing with the small categories, the 
isomorphism \eqref{eq:balmer} for $G$ a finite group scheme, and $R = k$ 
still a field, was established in  \cite{Friedlander/Pevtsova:2007a}. 
The first result on classifying localising ideals in a big 
category is due to Benson, Iyengar, and Krause \cite{Benson/Iyengar/Krause:2011a}; 
their classification recovers the case $\Rep(G,k)$ of 
Theorem~\ref{ithm:main} where $k$ is a field and $G$ is a finite group. 
This was generalised by the last four authors to finite group schemes 
in \cite{Benson/Iyengar/Krause/Pevtsova:2018a}. 

Moving to more recent history, and in fact to the motivating development 
for this work, we arrive at a theorem of Lau~\cite{Lau:2023a} who established 
the calculation \eqref{eq:balmer} for $G$ a finite group and $R$ any 
commutative noetherian ring. Hence, Theorem~\ref{ithm:main} generalises 
Lau's work in two directions: from finite groups to finite group schemes, 
and from $\rep(G,R)$ to its ind-completion $\Rep(G,R)$.Keeping in line with earlier work on representations over a field, both of 
these passages are far from routine, and require new ideas and methods.

For a start, while the objects in $\rep(G,R)$ can be described in simple terms and have been studied for long, as far as we know the category  $\Rep(G,R)$ has not appeared in the literature before, and even the objects in it are mysterious, when $R$ is not a regular ring. Our first task is to establish basic structural properties of $\Rep (G,R)$, and the paragraph below outlines some of our key findings.

\subsection*{The category $\Rep(G,R)$} Throughout $G$ is a finite flat group scheme over a noetherian commutative ring $R$. The group algebra $RG$ is a cocommutative Hopf $R$-algebra  that is a finitely generated and projective as an $R$-module. We consider representations of $G$ over $R$ which can be identified with $RG$-modules. 

Let $\Mod(G,R)$ be the category of $RG$-modules that are projective as $R$-modules, and $\rmod(G,R)$ its full subcategory consisting of the finitely generated modules, in other words, lattices. We write $\bfK(\Prj G, R)$ for the subcategory of $\KProj{G}$ consisting of complexes that are $K$-projective when restricted to $R$. For our purposes the more pertinent object is a subcategory of $\bfK(\Prj G, R)$ whose objects satisfy a local finiteness condition; this is the category $\Rep(G,R)$. When $R$ is regular, $\Rep(G,R)$ coincides with $\bfK(\Prj G)$, and this case is discussed in~\cite{Barthel:strat-regular, Benson/Iyengar/Krause/Pevtsova:2022b}. However, these categories differ for general rings. In Section~\ref{se:example} we discuss an example that illustrates this, based on the representation theory  of $\bbZ/2\times \bbZ/2$. 

The key features of these homotopy categories and the associated stable categories are described in Sections~\ref{se:Frobenius}--\ref{se:groups}, including a recollement  relating them; see Corollary \ref{co:recollement}. The constructions apply more generally to any Frobenius $R$-algebra  and the results of Section~\ref{se:Frobenius} are stated in that greater generality.

By way of motivation, we state fundamental properties satisfied by $\Rep(G,R)$ that make it a suitable homotopy category to consider: 
\begin{enumerate}
    \item 
    $\Rep(G,R)$ is a rigidly-compactly generated symmetric tensor triangulated category with product $\otimes_R$, unit $R$, and $\End^*(R)\cong  H^*(G,R)$. 
    \item 
The compact objects in $\Rep(G,R)$ identify with
\[
\rep(G,R)=\bfD^b(\rmod(G, R))\,,
\]
the bounded derived category of lattices.
\end{enumerate}

Inspired by work in~\cite{Lau:2023a}, Barthel~\cite{Barthel:strat}  proposed a homotopy theoretic model for $\Rep(G,R)$, proved that  it satisfies the properties above, and that it is stratified by $H^*(G, R)$, when $R$ is a Dedekind domain.  Subsequently, this result was generalised to all regular rings $R$, from a homotopical \cite{Barthel:strat-regular} and a homological \cite{Benson/Iyengar/Krause/Pevtsova:2022b} perspective, but the methodology of neither work readily generalised to arbitrary noetherian coefficients. The current work provides a model that unifies all these frameworks.

In  Section~\ref{se:enhancements} we describe how one can enhance the structure on  $\Rep(G, R)$ so it becomes a symmetric monoidal stable $\infty$-category. This allows us to compare the model of this paper to the $\infty$-category of derived representations of a finite group $G$ studied in \cite{Barthel:strat,Barthel:strat-regular}. The latter is a natural candidate suggested by the $\infty$-category context and that extends well to other ring spectra; see~\cite{Barthel/Castellana/Heard/Naumann/Pol}. We prove that the two categories are equivalent in the strongest possible sense, thereby reconciling the homological with the homotopical approach; see Proposition~\ref{prp:repmodel}. This section serves as motivation and as a bridge from modular representation theory to higher algebra; the rest of the paper is independent of it. 

Next we describe other key ingredients that go into the proof of Theorem~\ref{ithm:main}.

\subsection*{Reduction to fibres} 
As noted above, $\Rep(G, R)$ is a rigidly-compactly generated symmetric tensor triangulated category. Heuristically, it can be viewed as a sheaf of categories over $\Spec R$ with fibres $\Rep(G_{k(\fp)},k(\fp))$, for $\fp$ ranging over the prime ideals of $R$ and $G_{k(\fp)}$ denoting the  group scheme  base-changed along $R\to k(\fp)$.  This perspective is crucial for Lau's approach to computing the Balmer spectrum of $\bfD^b(\rmod (G,R))$ for a finite group $G$, and the approach to stratification in \cite{Benson/Iyengar/Krause/Pevtsova:2022b,Gnedin/Iyengar/Krause:2022a}.  

Let $X_{k(\fp)}\colonequals X \otimes_R k(\fp)$ denote the base change of an object $X \in \Rep(G, R)$ along $R \to k(\fp)$. The results below illustrate how the global structure of $\Rep(G,R)$ is controlled by its fibres.

\begin{introthm}[Theorem ~\ref{th:lg-kproj-ha}]
\label{ithm:fiberwise}
Let $G$ be a finite flat group scheme over a noetherian commutative ring $R$. The following statements hold.
    \begin{enumerate}[\quad\rm(1)]
        \item (Detection) For any $X \in \Rep(G, R)$ one has
        \[
        X=0 \quad\iff \quad \text{$X_{k(\fp)} = 0$ for all $\fp \in \Spec R$.}
        \]
        \item (Building) The following conditions are equivalent for $X,Y \in \Rep(G, R)$:
            \begin{enumerate}[\quad\rm(a)]
                \item $X\in  \Loc^{\otimes}(Y)$ in $\Rep(G, R)$;
                \item $\fibre X\fp\in \Loc^{\otimes}(\fibre Y{\fp})$
                  in $\Rep(G_{k(\fp)}, k(\fp))$ for each $\fp\in \Spec
                  R$. 
            \end{enumerate} 
    \end{enumerate}
\end{introthm}

As we explain in Section~\ref{se:fibrewisedetection}, `detection' holds for any Frobenius algebra; see Proposition~\ref{pr:conserve}. Part (2), given in Theorem~\ref{th:lg-kproj-ha} is a generalisation of \cite{Benson/Iyengar/Krause/Pevtsova:2022b} from regular to arbitrary noetherian coefficients; it immediately implies (1), but from a conceptual point of view  it is useful to separate these two conditions. 

Another crucial result is that the spectrum of the cohomology ring $H^*(G, R)$ is also controlled by the fibres.

\begin{introthm}[Theorem \ref{th:kappa}]
\label{ithm:fiberwise2}
For each $\fp$ in $\Spec R$ base change along $R \to k(\fp)$ induces a homeomorphism 
        \[
        \Spec H^*(G,k(\fp)) \longiso \Spec (H^*(G,R)\otimes_R k(\fp))\,.
        \]
\end{introthm}

The proof of this theorem occupies Section~\ref{se:cohomology-and-fibres}. The ideas that go into it are rather different from those involved in proving Theorem~\ref{ithm:fiberwise}. They build on earlier work of Benson and Habegger~\cite{Benson/Habegger:1987a}, and Lau~\cite{Lau:2023a} that establish the result for finite groups. The critical new inputs are recent results of van der Kallen~\cite{vanderKallen:2023a}, concerning torsion and finite generation of the cohomology ring $H^*(G,R)$ for a finite flat group scheme $G$ over $R$; see Section~\ref{ch:vanderkallen}. 

Theorems~\ref{ithm:fiberwise} and \ref{ithm:fiberwise2} allow us to employ the general theory of~\cite{Benson/Iyengar/Krause:2008a,Benson/Iyengar/Krause:2011a} for describing localising ideals in this context; we recall the relevant notions in Section~\ref{se:lch}.  One of the key insights of \cite{Benson/Iyengar/Krause/Pevtsova:2022b} is that stratification for regular rings may be checked fibrewise, thereby
reducing the problem to the case of field coefficients. After the preparation we have done so far, this now extends verbatim to arbitrary noetherian rings and finite group schemes, leading to the stratification part of Theorem~\ref{ithm:main}.  The costratification is then a formal
consequence, by combining the abstract bootstrap theorem of \cite{Barthel/Castellana/Heard/Sanders} with fibrewise detection; see Sections~\ref{se:stratification} and~\ref{se:costratification}. The results of van der Kallen mentioned above are again vital inputs in the arguments.

\begin{ack}
It is a pleasure to thank Greg Stevenson for helpful comments on a preliminary version of this paper. Part of this work was done during the Trimester Program ``Spectral Methods in Algebra, Geometry, and Topology" at the Hausdorff Institute in Bonn. It is a pleasure to thank HIM for hospitality and for funding by the Deutsche Forschungsgemeinschaft under Excellence Strategy EXC-2047/1-390685813.  The last four authors thank the American Institute for Mathematics for its hospitality and the support through the SQuaRE program. TB is supported by the European Research Council (ERC) under Horizon Europe (grant No.~101042990) and the Max Planck Institute for Mathematics, SBI was partly supported by NSF grant DMS-2001368, HK was partly supported by the Deutsche Forschungsgemeinschaft (SFB-TRR 358/1 2023 - 491392403), and JP was partly supported by NSF grants DMS-1901854, 2200832 and the  Brian and Tiffinie Pang faculty fellowship.
\end{ack}

\section{Frobenius algebras}
\label{se:Frobenius}
In this section we introduce the various module categories and homotopy categories of complexes that arise in our work. While our focus is on finite group schemes, some basic constructions and results apply more generally to Frobenius algebras. This is the content of this section. The methods build on those in \cite{Krause:2005a}; we take this work, and also  \cite{Krause:2022a}, as our standard references for the material below.

Throughout $R$ is a commutative ring (not necessarily noetherian) and $A$ a \emph{finite projective $R$-algebra}; that is to say, an $R$-algebra that is finitely generated and projective as an $R$-module. As usual $\Mod A$ stands for the category of all left $A$-modules, but we drop the term `left' in the sequel. We write $\Prj A$ for the full subcategory of projective $A$-modules. Our focus is on the subcategory of $\Mod A$ consisting of modules that are projective as $R$-modules; we denote it $\Mod(A,R)$. We write $\rmod(A,R)$ for the full subcategory of $\Mod(A,R)$ consisting of modules that are finitely generated over $R$.

\begin{lemma}
In $\Mod A$ the subcategory  $\Mod(A,R)$ is extension closed and so carries a natural exact structure, given by the short exact sequences of  $A$-modules.\qed
\end{lemma}

Let $\Prj(\Mod(A,R))$ be the full subcategory of projective objects in $\Mod(A,R)$, and $\Inj(\Mod(A,R))$ the full subcategory of injective objects.  The duality functor is an equivalence:
\begin{equation}
\label{eq:duality}
D\colonequals\Hom_R(-,R)\colon  (\rmod(A,R))^\op\longiso\rmod(A^\op,R)\,.
\end{equation}
For all $M\in\rmod(A,R)$ and  $N\in\rmod(A^\op,R)$ one has
\[
\Hom_A(M,DN)\cong\Hom_{A^\op}(N,DM)\,.
\]
It follows that the exact category $\rmod(A,R)$ has enough projective
and enough injective objects. We set 
\[
\omega_{A/R}\coloneqq\Hom_R(A,R)\,;
\]
this is the \emph{dualising bimodule} of the $R$-algebra $A$.

For any $A$-module $M$ let $\add M$ denote the full subcategory of $\Mod A$ consisting of direct summands of finite direct sums of copies of $M$, and $\mathrm{Add}\, M$ denotes the corresponding subcategory where all direct sums are permitted. When $M$ is in $\Mod(A,R)$ so is every module in $\mathrm{Add}\, M$.

\begin{lemma}\label{le:proj-inj}
  The exact category $\Mod(A,R)$ has enough projective and enough
  injective objects, with
  \[
\Prj (\Mod(A,R))=\mathrm{Add}\, A = \Prj A\quad\text{and}\quad \Inj(\Mod(A,R)) = \mathrm{Add}\, \omega_{A/R}\,.
\]
\end{lemma}

\begin{proof}
  We view  $\Prj R$ as an exact  category; in it all objects are projective and injective.  The
  forgetful functor $u\colon\Mod(A,R)\to\Prj R$ is exact. Thus its
  left adjoint $u_\lambda=A\otimes_R-$ preserves projectivity, while
  its right adjoint $u_\rho=\Hom_R(A,-)$ preserves injectivity. For
  $M\in \Mod(A,R)$ the counit $\eps_M\colon u_\lambda u(M)\to M$
  provides an epimorphism such that $u_\lambda u(M)$ is projective,
  while the unit $\eta_M\colon M\to u_\rho u(M)$ provides a
  monomorphism such that $u_\rho u(M)$ is injective. The map
  $u(\eps_M)$ is a split epimorphism, while $u(\eta_M)$ is a
  split monomorphism. Thus $\eps_M$ and $\eta_M$ provide exact
  sequences in $\Mod(A,R)$.

  The equality $\Prj (\Mod(A,R))=\Prj A$ is clear since $u_\lambda$ maps the category of free $R$-modules onto the category of free
  $A$-modules.
  
Given that $u_\rho$ preserves injectives and that it also preserves all direct sums, since $A$ is finitely generated as an $R$-module, it follows that the objects in $\Inj(\Mod(A,R))$ are precisely the direct  summands of direct sums of copies of $\omega_{A/R}$.
\end{proof}

We obtain a pair of functors
$\Sigma,\Omega\colon \Mod(A,R)\to \Mod(A,R)$ by completing, for
any module $M$ in $\Mod(A,R)$, the morphisms $\eta_M$ and $\eps_M$ to
exact sequences
\begin{equation}
\label{eq:Sigma-Omega}
\begin{gathered}
  0\lra M\xra{\eta_M} u_\rho u(M)\lra \Sigma M\lra 0 \\
  0\lra \Omega M\lra u_\lambda u(M)\xra{\eps_M} M\lra 0\,.
 \end{gathered}
\end{equation}
Iterating this process yields an injective and projective resolution of $M$, which we denote $\bfi M$ and $\bfp M$, respectively.

\subsection*{Derived categories}
We write $\bfD(\cat A)$ to denote the derived category of an exact category $\cat{A}$; see \cite[Section~4.1]{Krause:2022a}. The main examples for us are the exact category $\Mod(A,R)$ and its subcategory $\rmod(A,R)$.

\begin{proposition}\label{pr:derived-exact}
The inclusion $\Mod(A,R)\subseteq \Mod A$ of exact categories induces a quotient functor
  \[
\begin{tikzcd}
  \bfD( \Mod(A,R))\ar[r,twoheadrightarrow]&\bfD(\Mod A)
\end{tikzcd}
  \]
with kernel given by the complexes in $\Mod(A,R)$ that are acyclic in $\Mod A$.
\end{proposition}

\begin{proof}
Consider  the localising subcategory of $\bfD(\Mod(A,R))$ generated by $A$. Its right orthogonal in $\bfD(\Mod(A,R))$ is the full subcategory of complexes that are acyclic in $\Mod A$. It follows that the right adjoint of the inclusion is a quotient functor that identifies with the functor $\bfD(\Mod(A,R))\to\bfD(\Mod A)$ induced by the inclusion $\Mod(A,R)\to \Mod A$.
\end{proof}

In what follows, by a \emph{perfect} complex over a ring $\Lambda$ we mean a complex that is quasi-isomorphic to a bounded complex of finitely generated projective $\Lambda$-modules.

\begin{proposition}
\label{pr:cofinal}
The inclusions $\rmod(A,R)\subseteq \Mod(A,R)\subseteq \Mod A$ of exact categories induces functors 
\[
\bfD^b(\rmod(A,R))\lra\bfD(\Mod(A,R)) \lra\bfD(\Mod A)
\]
where the first one and the composite are fully faithful, identifying $\bfD^b(\rmod(A,R))$ with the subcategory of complexes in $\bfD(\Mod A)$ that are perfect over $R$. 
\end{proposition}

\begin{proof} 
The composite of functors in question and also the first one are fully faithful, because $\rmod(A,R)$ is
\emph{cofinal} in $\Mod A$. This means for any exact sequence 
\[
0\lra X\lra Y\lra Z\lra 0\qquad \text{with}\qquad Z\in\rmod(A,R)
\]
there is an epimorphism $A\otimes_R Z\to Z$ in $\rmod(A,R)$ that factors through $Y\to Z$; see \cite[Proposition~4.2.15]{Krause:2022a}.  It is clear that the essential image consists of complexes in $\bfD(\Mod A)$ that are perfect over $R$. As to the converse claim, let $X$ be a complex of $A$-modules that is perfect over $R$. We claim that $X$ is quasi-isomorphic to a bounded complex consisting of $A$-modules that are finite and projective over $R$; this would complete the argument for the latter is evidently in the image of $\bfD^b(\rmod(A,R))$.

Indeed, since $X$ is perfect over $R$ its homology modules $H_i(X)$ are finitely generated as $R$-modules, and so also as $A$-modules, and zero for $|i|\gg 0$. Without loss of generality we can assume $H_i(X)=0$ for $i<0$. Thus $X$ has a resolution $P$ over $A$ where each $P_i$ is finitely generated projective and zero for $i<0$. For each integer $n$ set $\Omega_n \colonequals \Coker(P_{n+1}\to P_n)$. Then $X$ is quasi-isomorphic to the complex 
\[
X(n)\colon \quad \cdots\lra 0 \lra \Omega_{n} \lra P_{n-1}\lra \cdots\lra P_0\lra 0 \lra\cdots\
\] 
for each $n\ge s\colonequals\max\{i\mid H_i(X)\ne0\}$. 
The $A$-module $\Omega_s$ is perfect over $R$ and let $p$ denote its projective dimension. Then the module $\Omega_{s+p}$ is projective over $R$, and therefore the complex $X(s+p)$ is quasi-isomorphic to $X$ and in $\rmod(A,R)$.
\end{proof}

The following example has been suggested by Greg Stevenson in order to illustrates the difference between the derived categories of $\Mod(A,R)$ and $\Mod A$.

\begin{example}
    Let $k$ be a field and choose $A\colonequals k[t]/(t^2)\equalscolon R$. Then $\bfD(\Mod(A,R))$ equals the homotopy category $\bfK(\Prj A)$ of complexes of projective $A$-modules, and the kernel of the canonical functor $\bfD(\Mod(A,R))\to\bfD(\Mod A)$ identifies with the stable module category of $A$. In particular, a complex is in the kernel if and only if it is a direct sum of copies of the complex
\[
\cdots\xra{\ t\ }A\xra{\ t\ }A\xra{\ t\ }A\xra{\ t\ }\cdots\,.
\]\end{example}

\subsection*{Frobenius algebras}
We are interested in the case when the exact category $\Mod(A,R)$ is \emph{Frobenius}, that is to say, it has enough projective and enough injective objects, and these coincide.

\begin{lemma}
\label{le:Frobenius}
Let $R$ be a commutative ring and $A$ a finite projective $R$-algebra. The following conditions are equivalent.
  \begin{enumerate}[\quad\rm(1)]
\item There is an equality $\add \omega_{A/R}=\add A$.
\item The exact category  $\rmod(A,R)$ is a Frobenius category. 
\item The exact category  $\Mod(A,R)$ is a Frobenius category.
\end{enumerate}
\end{lemma}

\begin{proof}
The equivalence (1) $\Leftrightarrow$ (2) is clear since $\add \omega_{A/R}$ equals the category of injective objects in
$\rmod(A,R)$, while $\add A$ equals the category of projectives.

(1) $\Leftrightarrow$ (3) follows from Lemma~\ref{le:proj-inj}.
\end{proof}

A ring $A$ is a \emph{Frobenius $R$-algebra} if it is a finite projective $R$-algebra  satisfying the conditions in the result above; see the work of Scheja and Storch~\cite[\S14]{Scheja/Storch:1974a}, where this notion is introduced via condition (1).  When $R$ is noetherian, this condition can be detected in terms of the fibres of $A$ over $R$; see Lemma~\ref{le:frobenius-self-injective}. Evidently $A$ is a Frobenius $R$-algebra if and only if so is $A^\op$, because of the duality \eqref{eq:duality}. 

\subsection*{The stable module category}

For any Frobenius category $\cat A$ we write $\St\cat A$ for the
\emph{stable category}, obtained from $\cat A$ by annihilating
the projective and injective objects; this carries a natural
triangulated structure~\cite[Section~3.3]{Krause:2022a}. We denote $\Sigma$ the suspension in
$\St\cat A$, which for any $M\in\cat A$ is given by an exact sequence
\[
0\lra M\lra Q\lra \Sigma M\lra 0
\] 
such that $Q$ is an injective, and so also a projective, object. For $M,N$ in $\cat A$ we set
\[
\underline{\Hom}_{\cat A}(M,N)\colonequals \Hom_{\cat A}(M,N)/\mathrm{PHom}_{\cat A}(M,N)\,,
\]
where $\mathrm{PHom}_{\cat A}(M,N)$ denotes the subgroup of morphisms $M\to N$ that
factor through a projective object.

Let $A$ be a Frobenius $R$-algebra.  We set
\[
\stmod(A,R)\colonequals\St(\rmod(A,R))\quad\text{and}\quad
  \StMod(A,R)\colonequals\St(\Mod(A,R))\,,
  \]
and consider these categories with their natural triangulated structure.

\subsection*{The homotopy category of projectives}
Let $\Lambda$ be a ring and $\bfK(\Prj \Lambda)$ its homotopy
category of (unbounded) complexes of projective $\Lambda$-modules. Consider the canonical functor $\bfq\colon \bfK(\Prj \Lambda)\to\bfD (\Mod \Lambda)$ which inverts the quasi-isomorphisms. Taking projective resolutions yields a left adjoint 
adjoint $X\mapsto \bfp X$ that identifies the derived category $\bfD (\Mod \Lambda)$ with the full subcategory of \emph{$K$-projective
  complexes} in $\bfK(\Prj \Lambda)$.

Let $A$ be a Frobenius $R$-algebra. We write $\bfK(\Prj A,R)$ for the full subcategory of complexes in $\bfK(\Prj A)$ that are $K$-projective when restricted to $R$; this is a localising subcategory. Let $\mathbf{Ac}(\Prj A,R)$ denote the full subcategory of complexes in $\bfK(\Prj A,R)$ that are acyclic; this is again a localising subcategory. The term `acyclic' refers to the exact structure of the ambient exact category, which is either $\Mod A$ or $\Mod(A,R)$. In both cases we obtain the same class of acyclic complexes. It is clear that a complex is acyclic in $\Mod A$ if it is acyclic in $\Mod(A,R)$. On the other hand, if a complex is acyclic in $\Mod A$ and also $K$-projective over $R$, then it is contractible over $R$ and therefore acyclic in $\Mod(A,R)$.

Recall that each $M$ in $\Mod(A,R)$ has an injective and a projective resolution, denoted $\bfi M$ and $\bfp M$; see \eqref{eq:Sigma-Omega}. A \emph{complete resolution} of an $M$  in $\Mod(A,R)$ is an acyclic complex $\bft M$ of projective $A$-modules with $ \Coker(d^{-1}_{\bft M})=M$; here, $d^{-1}_{\bft M}$ denotes the $(-1)$st differential in the complex $\bft M$. Then $\bfi M$ and $\bft M$ are $K$-projective over $R$, since $A$ is a Frobenius $R$-algebra; the complex $\bfp M$ is always $K$-projective over $R$.

\begin{lemma}
\label{le:KProj-ac}
When $A$ is a Frobenius $R$-algebra the assignment $X\mapsto \Coker(d^{-1}_X)$ induces an $R$-linear triangle equivalence
\[
\mathbf{Ac}(\Prj A,R)\longiso\StMod(A,R)\,.
\]
\end{lemma}

\begin{proof}
Each $X$ in $\mathbf{Ac}(\Prj A,R)$ is contractible when restricted to $R$. Thus the assignment
  $X\mapsto \Coker(d^{-1}_X)$ is well-defined.  The quasi-inverse
  takes $M\in\Mod(A,R)$ to a complete resolution $\bft M$ in
$\Prj A$.
\end{proof}

\begin{proposition}\label{pr:KProj-loc-seq}
When $A$ is a Frobenius $R$-algebra the canonical functor
 \[
 \bfK(\Prj A,R)\lra\bfD (\Mod A)
 \]
 induces a localisation sequence
\begin{equation*}
\begin{tikzcd}
\mathbf{Ac}(\Prj A,R) \arrow[hookrightarrow, yshift=-.75ex]{rr} && \bfK(\Prj A,R)
\arrow[twoheadrightarrow,yshift=.75ex]{ll}[swap]{\bft}
  \arrow[twoheadrightarrow, yshift=-.75ex]{rr} &&\dcat A
   \arrow[tail,yshift=.75ex]{ll}[swap]{\bfp }
\end{tikzcd}
\end{equation*}
\end{proposition}

\begin{proof}
The left adjoint $\bfp \colon \bfD(\Mod A) \to \bfK(\Prj A)$ clearly lands in $\bfK(\Prj A,R)$. For any $X\in \bfK(\Prj A)$, completing the counit $\bfp  X\to X$ to  an exact triangle 
\[
\bfp  X\lra X\lra \bft X\lra
\]
provides the left  adjoint $\bft$ for the inclusion of $\mathbf{Ac}(\Prj A,R)$.
\end{proof}

\subsection*{Compact generation}

Let $A$ be a Frobenius $R$-algebra. We write $\lfMod(A,R)$ for the localising subcategory of $\Mod(A,R)$ generated by the modules in $\rmod(A,R)$. By a \emph{localising
  subcategory} we mean a full subcategory closed under all coproducts and
satisfying the \emph{two-out-of-three property}: for any exact sequence
\[
0\lra X\lra Y\lra Z\lra 0
\]
if any two of $X,Y$, or $Z$ is in the subcategory so is the third.  The subcategory $ \lfMod(A,R)$ contains all the projective objects in $\Mod(A,R)$, and so is itself a Frobenius category, with induced exact structure. When $R$ is regular $ \lfMod(A,R)= \Mod(A,R)$, by Lemma~\ref{pr:KProj-regular}, but not always; see Section~\ref{se:example}. Set
\[
\StlfMod(A,R) \colonequals \St(\lfMod(A,R))\,.
\]
Here are the key properties of this triangulated category.

\begin{proposition}
\label{pr:StMod}
Let $A$ be a Frobenius $R$-algebra. The  triangulated category $\StlfMod(A,R)$ is compactly generated and its
  subcategory of compact objects identifies with the idempotent
  completion of $\stmod(A,R)$.
\end{proposition}

\begin{proof}
  The canonical functor $ \lfMod(A,R)\to\StlfMod(A,R)$ preserves coproducts and hence  $\stmod(A,R)$ consists of  compact objects. They generate $\StlfMod(A,R)$ as a triangulated category by definition so an arbitrary  compact object is a direct summand of an object in  $\StlfMod(A,R)$; see \cite[Proposition~3.4.15]{Krause:2022a}.
\end{proof}

\begin{chunk}
The category $\stmod(A,R)$ is idempotent complete when  $R$ is a complete local ring. This need not be the case for a general $R$.

Indeed, the local ring $A\colonequals \mathbb{C}[x,y]_{(x,y)}/(x^2-y^2(y+1))$, viewed as an algebra over $R\colonequals \mathbb{C}[x]_{(x)}$, is a Frobenius $R$-algebra; this can be checked easily. Since $R$ is a DVR, a finitely generated module over it is projective (equivalently, free) if and only it is torsion-free. Thus an $A$-module is in $\rmod(A,R)$ if and only if it has positive depth, that is to say, it is maximal Cohen--Macaulay. Thus $\stmod(A,R)$ is the stable category of maximal Cohen--Macaulay $A$-modules. This category is not idempotent complete; see \cite[Lemma~6.2.12]{Krause:2022a} for a proof.

For an example where $A$ is the group algebra of a finite group, see forthcoming work of J.~Grodal and A.~Krause.
\end{chunk}

Set $\cat A\colonequals \Mod(A,R)$. For a full additive subcategory $\cat C\subseteq\cat A$ we write $\bfK(\cat C)$ for the homotopy category of chain complexes. We set
\[
\bfK^{+,b}(\cat C)=\{X\in\bfK^+(\cat C)\mid H^nX=0 \textrm{ for }|n|\gg 0\}
\]
where $\bfK^+(\cat C)$ denotes the complexes $X$ with $X^n=0$ for $n\ll 0$, and the condition $H^nX=0$ means that $d_X^{n-1}$ can be written as the
composite 
\[
X^{n-1}\twoheadrightarrow\Ker d^n_X\rightarrowtail X^n
\] 
of an admissible epimorphism and an admissible monomorphism in $\cat A$.  In
particular, the subcategory $\bfK^{+,b}(\cat C)$ depends on the ambient
category $\cat A$, even though it is not part of the notation.
The subcategory  $\bfK^{-,b}(\cat C)$ is defined analogously.

\begin{lemma}
\label{le:KProj-compact}
Let $A$ be a Frobenius $R$-algebra. For an object $X$ in $\bfK(\Prj A)$ the following conditions are equivalent.
\begin{enumerate}[\quad\rm(1)]
\item $X$ is compact  in $\bfK(\Prj A)$ and $K$-projective over $R$.
\item $X$ fits into an extension $0\to X'\to X\to X''\to 0$ such that
  $X'=\bfi M$ for some $M\in\rmod(A,R)$ and $X''$ is perfect.
\item $X$ is in $\bfK^{+,b}(\prj A)$.
\end{enumerate}
Moreover, such an $X$ is perfect over $R$.
\end{lemma}

\begin{proof}
  (1)$\Rightarrow$(2)  Let $X$  be compact and $K$-projective over $R$. The restriction functor $\bfK(\Prj A)\to \bfK(\Prj R)$  preserves compactness, since its right adjoint $\Hom_R(A,-)$  preserves coproducts.  Thus $X$ is compact over $R$. Since it is also $K$-projective over $R$ we deduce that it is perfect over $R$. So $X$ belongs to $\bfK^+(\prj A)$ by \cite[Proposition~7.6]{Neeman:2008a},
  and $H^i(X)=0$ for all $i\ge n$ and some $n\in\bbZ$. We truncate in degree $n$ and write $X$ as an
  extension $0\to X'\to X\to X''\to 0$ such that $X'$ is concentrated
  in degrees $i \ge n$ and $X''$ is concentrated in degrees $i
  <n$. Clearly, $X''$ is perfect and hence $X'$ is $K$-projective
  over $R$. Splicing $X''$ together with a projective resolution of
  $M\coloneqq\Ker d^n_X$, yields a complex over $R$ that is acyclic and
  $K$-projective, and hence contractible. Thus
  $X''=\bfi M$ with $M\in\rmod(A,R)$.

  (2)$\Rightarrow$(1) It suffices to verify that $X''$ and $X'$ are compact and $K$-projective over $R$. Clearly $X''$ has these properties. Also $X'$ is
  $K$-projective over $R$.  For any $Y$ in $\bfK(\Prj A)$ one has an isomorphism
  \[
  \Hom_{\bfK(A)}(\bfi M,Y)\cong \Hom_{\bfK(A)}(M,A)\,,
  \] 
  by \cite[Lemma~2.1]{Krause:2005a}, so $X'$ is compact. 

  (2)$\Leftrightarrow$(3) This is clear. 
  
  The last assertion was verified in proving (1)$\Rightarrow$(2).
\end{proof}

We write $\lfK(\Prj A,R)$ for the localising subcategory of $\bfK(\Prj A)$ that is generated by the objects
satisfying the equivalent conditions in Lemma~\ref{le:KProj-compact}. Thus one has inclusions
\[
\lfK(\Prj A,R) \subseteq \bfK(\Prj A,R) \subseteq \bfK(\Prj A)
\]
which become equalities when the ring $R$ is regular; cf.\ Proposition~\ref{pr:KProj-regular}. In fact, it is not hard to verify that equality holds on the right only if $R$ is regular. Equality holds on the left when, for instance, $A=R$, and also in some other cases, but not always; see Section~\ref{se:example}.

\begin{proposition}
\label{pr:KProj}
Let $A$ be a Frobenius $R$-algebra. The triangulated category $\lfK(\Prj A,R) $ is compactly generated and the canonical functor
  $\lfK(\Prj A,R)\to\bfD (\Mod A)$ induces a triangle
  equivalence
  \[
\lfK(\Prj A,R)^{c}\longiso\, \bfD^b(\rmod(A,R))\,.
  \]
\end{proposition}

\begin{proof}
  The triangulated category $\lfK(\Prj A,R)$ is compactly generated by
  construction.  For the description of the compact objects recall
  that $\rmod(A,R)$ is an exact category with enough injective
  objects. Thus we have a triangle equivalence
  $\bfK^+(\prj A)\iso \bfD^+(\rmod(A,R))$ which restricts to an
  equivalence 
  \[
\lfK(\Prj A,R)^{c}=\bfK^{+,b}(\prj A)\iso \bfD^b(\rmod(A,R)).\qedhere
  \]
\end{proof}

We write $\bfD^{\mathrm{perf}}(A)$ for the category of perfect complexes, which identifies with the full subcategory
of compact objects in $\bfD(\Mod A)$.

\begin{corollary}
\label{co:recollement}
Let $A$ be a Frobenius $R$-algebra. Then the canonical functor $\lfK(\Prj A,R)\to\bfD (\Mod A)$ induces a recollement
\begin{equation*}
\label{eq:recollement}
\begin{tikzcd}
\StlfMod(A,R) \arrow[tail]{rr} && \lfK(\Prj A,R)
\arrow[twoheadrightarrow,yshift=1.5ex]{ll}[swap]{\bft}
\arrow[twoheadrightarrow,yshift=-1.5ex]{ll}
  \arrow[twoheadrightarrow]{rr} &&\dcat A
   \arrow[tail,yshift=1.5ex]{ll}[swap]{\bfp }
   \arrow[tail,yshift=-1.5ex]{ll}
\end{tikzcd}
\end{equation*}
and the pair of left adjoints $(\bfp ,\bft)$  induces (when restricted
to compact objects) a triangle equivalence
\[
\bfD^b(\rmod(A,R))/\bfD^{\mathrm{perf}}(A)\longiso\stmod(A,R)\, .
\]
\end{corollary}

\begin{proof}
Recall that $\bfp$ is left adjoint to $\bfq\colon \lfK(\Prj A,R) \to\bfD (\Mod A)$. The latter preserves coproducts so it admits a right adjoint, by Brown representability. This yields the right half of the recollement; the left half is a consequence. 
  
  Indeed, we have the full subcategory of acyclic complexes in $\lfK(\Prj A,R)$ on the left, that identifies with a localising subcategory of $\StMod(A,R)$ by  Lemma~\ref{le:KProj-ac}. This subcategory is generated by the modules from
  $\rmod(A,R)$ since $\lfK(\Prj A,R)$ is generated by their injective
  resolutions, and therefore it equals $\StlfMod(A,R)$. 

As to the description of the compact objects, $\bfq$ preserves coproducts so the left adjoint $\bfp $ preserves compactness. The compacts in $\bfD(\Mod A)$ identify with the category $\bfD^{\mathrm{perf}}(A)$ of perfect  complexes, while the compacts in $\lfK(\Prj A,R)$ identify with
  $\bfD^b(\rmod(A,R))$ by Proposition~\ref{pr:KProj}. This yields the
  last assertion, since for any Frobenius category $\cat A$ with full
  subcategory of projectives $\cat P$ we have
  \[
  \bfD^b(\cat A)/\bfD^b(\cat P)\longiso\St\cat A\,,
  \]
  for example by \cite[Proposition~4.4.18]{Krause:2022a}. 
\end{proof}

\subsection*{Regular base}
A commutative ring $R$ is \emph{regular} if it is noetherian and the local ring $R_\fp$ has finite global dimension for each $\fp$ in $\Spec R$; when the Krull dimension of $R$ is finite, $R$ is regular if and only if its global dimension is finite; see \cite[Section~2.2]{Bruns/Herzog:1998a}. In what follows, we often use the following characterisation of regularity: A  noetherian commutative ring  $R$ is regular if and only if every complex of projective $R$-modules is $K$-projective; equivalently, every acyclic complex of projective modules is contractible; see, for instance, \cite[Section~3]{Iacob/Iyengar:2009a}.

\begin{proposition}
\label{pr:KProj-regular}
  Let $R$ be a regular ring and $A$ a Frobenius $R$-algebra. There are equalities
  \[
   \lfMod(A,R)=\Mod(A,R)\qquad\text{and}\qquad \lfK(\Prj A,R)=\bfK(\Prj A)\,.
 \]
Moreover, there is an equality $\bfD(\Mod(A,R))=\bfD(\Mod A)$.
\end{proposition}

\begin{proof}
  By \cite[Proposition~7.14]{Neeman:2008a}, the triangulated category $\bfK(\Prj A)$ is compactly generated.  From Lemma~\ref{le:KProj-compact} and regularity of $R$, it follows that each compact object
  belongs to $\lfK(\Prj A,R)$. Thus $ \lfK(\Prj A,R)=\bfK(\Prj A)$.
Comparing the diagrams in Proposition~\ref{pr:KProj-loc-seq} and
Corollary~\ref{co:recollement} and using Lemma~\ref{le:KProj-ac}, we conclude that $\StlfMod(A,R)=\StMod(A,R)$. Therefore, any $M \in \Mod(A,R)$ may be represented by an object in $\lfMod(A,R)$ up to a projective summand. Since projectives are in $\lfMod(A,R)$, it follows that  $\lfMod(A,R)=\Mod(A,R)$. The last equality follows from Proposition~\ref{pr:derived-exact} once one observes that for any complex in $\Mod (A,R)$ being acyclic in $\Mod A$ implies contractible in $\Mod R$ and therefore also acyclic in $\Mod(A,R)$.
\end{proof}

\section{Fibrewise detection}
\label{se:fibrewisedetection}

Let $R$ be a noetherian commutative ring and $A$ a finite projective $R$-algebra. In this section we prove that certain homological properties of $A$-modules can be detected locally. To that end, for each $\fp$ in $\Spec R$, let $k(\fp)$ denote the residue field of $R$ at $\fp$, and for any $A$-complex $X$, set
\[
  \fibre X{\fp} \colonequals X\otimes_R k(\fp)\,.
\]
The assignment $X\mapsto \fibre X{\fp}$ induces exact \emph{fibre functors}
\begin{align*}
\Mod(A,R) &\lra \Mod_{k(\fp)} A_{k(\fp)} =\Mod A_{k(\fp)} \\
\bfK(\Prj A,R) &\lra \bfK(\Prj A_{k(\fp)},k(\fp))=\bfK(\Prj A_{k(\fp)})\,.
\end{align*}
We prove that these functors are \emph{jointly conservative} when
$\fp$ runs through $\Spec R$; see Proposition~\ref{pr:conserve}. This requires some preparations.

\begin{lemma}
\label{le:Deligne}
  Let $R$ be a noetherian commutative ring, $A$ a flat $R$-algebra,
  and $F$ an $A$-module that is flat over $R$. If for each
  $\fp \in \Spec R$ the $A_{k(\fp)}$-module $F_{k(\fp)}$ is flat, then the
  $A$-module $F$ is flat.
\end{lemma}

\begin{proof}
The argument  follows Deligne's proof of
\cite[Prop. 5.5.23]{Deligne:1972b}.

We begin with a general observation. For any ideal $I$ of $R$ we write
\[
V(I)=\{\fp\in\Spec R\mid I\subseteq \fp\}
\] 
for the corresponding closed subset. Then for $V=V(I)$ and any $A$-module $M$ its submodule
of sections 
\[
\gam_V M=\{x\in M\mid I^n x=0\text{ for some }n>0\}
\]
equals the maximal $A$-submodule $N\subseteq M$ \emph{supported on $V$}, that is
$N_\fp=0$ for $\fp\not\in V$. Note that $M=\colim_V\gam_V M$ where
$V$ runs through all closed subsets of $\Spec R$.

We prove by noetherian induction on closed subsets $V$ of $\Spec R$
that for each (right) $A$-module $M$ supported on $V$ one has
\[
\Tor^A_i(M,F) =0 \quad\text{for $i\ge 1$.}
\]
Say $V=V(I)$ for some reduced ideal $I$ in $R$. For each integer $n\ge 0$ the subset 
\[
M_n \colonequals \{x\in M\mid I^nx=0\}
\]
is an $A$-submodule of $M$, and since $M$ is supported on $V$ we get that
$M=\bigcup_{n\ge 0} M_n$.
Since $\Tor^A_i(M,F) = \colim_n \Tor^A_i(M_n,F)$, it suffices to verify the desired vanishing for $M$ such that $I^n M=0$. Then $M$ has a finite filtration 
\[
0 = I^n M \subseteq I^{n-1} M \subseteq \cdots \subseteq IM \subseteq M
\]
by $A$-submodules and each subquotient is annihilated by $I$. Thus we can replace $M$ by one of these subquotients and assume $IM=0$. Let $\fp$ be a prime ideal in $R$ minimal in $V(I)$, and consider the exact sequence
\begin{equation}
\label{eq:4-term}
0\lra M' \lra M \xra{\ \iota\ } M_{k(\fp)} \lra M''\lra 0.
\end{equation}
Observe that $\iota$ is $A$-linear so the exact sequence is one of
$A$-modules. Since $I$ is reduced and $\fp$ is minimal over $I$, the
map $\iota_\fp$ is bijective, and hence $M'_\fp=0=M''_\fp$. In
particular $M'$ and $M''$ can be expressed  as filtered colimits of $A$-submodules supported on proper closed subsets $W\subset V$. Thus the induction hypothesis yields
\[
\Tor^A_i(M',F)=0 = \Tor^A_i(M'',F)\quad \text{for $i\ge 1$.}
\]
Moreover since $A$ and $F$ are flat over $R$ one has the first isomorphism below
\[
\Tor^A_i(M_{k(\fp)}, F) \cong \Tor^{A_{k(\fp)}}_i(M_{k(\fp)},F_{k(\fp)}) = 0 \quad\text{for $i \ge 1$.}
\]
The equality holds by the hypotheses on the flatness of $F_{k(\fp)}$. These computations and \eqref{eq:4-term} yield the desired vanishing.
\end{proof}

A ring $\Lambda$ is called \emph{self-injective} if injective and projective $\Lambda$-modules coincide. Then the opposite ring
$\Lambda^\op$ is also self-injective and the ring is artinian; cf.~\cite[Example~3.3.4]{Krause:2022a} for details. The importance of the result below is that, for algebras finite projective over a noetherian commutative ring, the Frobenius property is detected on fibres. In \cite{Benson/Iyengar/Krause/Pevtsova:2022b} such an $A$ is said to be \emph{fibrewise self-injective}; in this work, we prefer to call them Frobenius algebras, in view of Lemma~\ref{le:Frobenius}.

\begin{lemma}
\label{le:frobenius-self-injective}
Let $R$ be a noetherian commutative ring.  A finite projective $R$-algebra $A$ is Frobenius if and only if $\fibre A{\fp}$ is self-injective for each $\fp$ in $\supp_R A$.
\end{lemma}

\begin{proof}
Since $A$ is finite projective over $R$,  for each $\fp$ in  $\Spec R$ one has an isomorphism of $\fibre A{\fp}$-bimodules:
\[
\omega_{A/R}\otimes_R k(\fp) \cong \omega_{{\fibre A{\fp}}/k(\fp)}\,.
\] 
An algebra $\Lambda$ that is finite dimensional over a field $k$ is self-injective if and only if $\add \omega_{\Lambda/k}=\add \Lambda$. Thus the isomorphism above and Lemma~\ref{le:Frobenius} yield that when $A$ is Frobenius, the ring $\fibre A{\fp}$ is self-injective. 

Suppose $\fibre A{\fp}$ is self-injective for each $\fp$ in $\Spec R$. Then the $\fibre A{\fp}$-module $\omega_{A/R}\otimes_R k(\fp)$ is projective for each $\fp$, by the isomorphism above.  For finitely presented modules flatness and projectivity agree, so Lemma~\ref{le:Deligne} yields that $\omega_{A/R}$ is a projective $A$-module, that is to say,  it is in $\add A$. Using duality, we deduce that 
$A$ is in $\add \omega_{A/R}$. Thus the $R$-algebra   $A$ is Frobenius, by Lemma~\ref{le:Frobenius}.
\end{proof}

\begin{lemma}
\label{le:syzygies}
Let $R$ be a noetherian commutative ring, $A$ a  Frobenius $R$-algebra, and $M$ an $A$-module that is projective over $R$. If $M$ is flat over $A$, then so are the $A$-modules $\Sigma M$ and $\Omega M$.
\end{lemma}

\begin{proof}
Clearly, when the $A$-module $M$ is finitely generated, if it is projective, so are $\Sigma M$ and $\Omega M$. A module is flat if and only if it is a filtered colimit of finitely generated projective modules.  Thus the assertion follows from the fact that the functors $\Sigma$ and $\Omega$ defined via the exact sequences \eqref{eq:Sigma-Omega} preserve filtered colimits.
\end{proof}

\begin{proposition}\label{pr:flat}
Let $R$ be a noetherian commutative ring and $A$ a  Frobenius $R$-algebra. For each $M\in \Mod(A,R)$ the following conditions are equivalent.
\begin{enumerate}[\quad\rm(1)]
\item $M$ is a flat $A$-module. 
\item $M$ is a projective $A$-module. 
\item $M_{k(\fp)}$ is a projective $A_{k(\fp)}$-module for each $\fp\in\Spec R$. 
\item $M_{k(\fp)}$ is a flat $A_{k(\fp)}$-module for each $\fp\in\Spec R$. 
\end{enumerate}
\end{proposition}

\begin{proof}
  (1) $\Rightarrow$ (2) It follows from Lemma~\ref{le:syzygies}
  that the flat modules in $\Mod(A,R)$ are closed under the
  suspension $\Sigma$ and its stable inverse $\Sigma^{-}\colonequals\Omega$. 
  For  each $n\in\bbZ$ we obtain an exact sequence
  $0\to \Sigma ^n M\to P_n\to \Sigma ^{n+1} M\to 0$ such that $P_n$ is
  projective, by setting
  \[\Sigma ^n M\colonequals
    \begin{cases}
      \Sigma(\Sigma^{n-1} M)&n>0,\\
      M&n=0,\\
      \Sigma^-(\Sigma^{n+1} M)&n<0.
\end{cases}\] These yield an exact sequence
\[0\lra \bigoplus_{n\in\mathbb Z}\Sigma ^n M\lra
  \bigoplus_{n\in\mathbb Z} P_n\lra \bigoplus_{n\in\mathbb Z} \Sigma
  ^{n} M\lra 0\,.\] Thus the module
$\bar M= \bigoplus_{n\in\mathbb Z}\Sigma ^n M$ is flat and periodic. Then it follows from
\cite[Theorem~2.5]{Benson/Goodearl:2000a} that $\bar M$ projective, and in
particular $M=\Sigma^0 M$ is projective.

(2) $\Rightarrow$ (3) $\Rightarrow$ (4) Clear.  

(4) $\Rightarrow$ (1) This follows from Lemma~\ref{le:Deligne}.
\end{proof}

\begin{proposition}\label{pr:conserve}
Let $R$ be a noetherian commutative ring and $A$ a  Frobenius $R$-algebra. For an object $X$
  in $ \bfK(\Prj A,R)$ we have $X=0$  if and only if $X_{k(\fp)} = 0$
  in $\bfK(\Prj A_{k(\fp)})$ for all $\fp \in \Spec R$.
\end{proposition}

\begin{proof}
The `only if' direction is clear. So suppose $X_{k(\fp)} = 0$ for all
  $\fp \in \Spec R$.  Applying the functor $\bfp $ from
  Proposition~\ref{pr:KProj-loc-seq} we obtain an object $\bfp X$ in
  the derived category $\dcat A$ such that $\bfp X \otimes_R k(\fp)$ is
  acyclic for all $\fp \in \Spec R$. Considering $\bfp X$ as an
  object in $\dcat R$ and applying a result of Neeman
  \cite[Lemma 2.12]{Neeman:1992a} we conclude that $\bfp (X)$ is acyclic. Hence
  $X$ is homotopy equivalent to an object in $\mathbf{Ac}(\Prj A,R)$.
  Since the latter is equivalent to $\StMod(A,R)$ by
  Lemma~\ref{le:KProj-ac}, the statement now follows from
  Proposition~\ref{pr:flat}.
\end{proof}

\section{Finite flat group schemes}
\label{se:groups}

From now on $R$ is a  noetherian commutative ring and $G$ a finite flat group scheme over $R$. The coordinate algebra $R[G]$ is a finitely generated projective $R$-module and a commutative Hopf $R$-algebra. The group algebra $RG = \Hom_R(R[G],R)$ is then  a finitely generated projective $R$-module and a cocommutative Hopf algebra. Hence, $RG$ is a Frobenius $R$-algebra, so the results of the previous sections apply. Examples of finite flat groups schemes are given  at the end of this section.

There is an equivalence of categories between $R$-linear representations of $G$ and $RG$-modules. Since $RG$ is a Hopf algebra there is more structure associated to its various module categories. This is the topic of this section.  In what follows, by a $G$-module we mean an $RG$-module, and we write $\rmod(G,R)$ for $\rmod(RG,R)$,   $\Prj G$ instead of $\Prj RG$, and so forth.

\subsection*{Tensor structure}
Given $G$-modules $X$ and $Y$, there is a natural diagonal $G$-module structure on $X\otimes_RY$, obtained by restricting its
$(G\times G)$-module structure along the coalgebra map $\Delta\colon RG\to RG\otimes_RRG$.

\begin{lemma}
Let $P,Q$ be $G$-modules. If the $G$-module $P$ is projective and $Q$ is projective over $R$, then the $G$-module $P\otimes_RQ$ is projective.
\end{lemma}

\begin{proof}
This follows from the standard adjunction isomorphism 
\[
\Hom_{G}(P\otimes_RQ,-)\cong \Hom_G(P, \Hom_R(Q,-))\,. \qedhere
\]
\end{proof}

The preceding result implies that $\otimes_R$ induces a tensor product on  $\bfK(\Prj G)$, the homotopy category of projective $G$-modules.

\begin{lemma}
\label{le:tensor-unit}
The triangulated category $\bfK(\Prj G)$ is tensor triangulated, with product $\otimes_R$ and unit the injective resolution $\bfi R$ in $\rmod(G,R)$.
\end{lemma}

\begin{proof}
We have seen that $\otimes_R$ provides a tensor product  on $\bfK(\Prj G)$. It remains to verify the assertion about the unit; for this see \cite[Proposition~5.3]{Benson/Krause:2008a}, whose proof extends to finite flat group schemes, and also \cite[Lemma~4.2]{Benson/Iyengar/Krause/Pevtsova:2022b}.
 \end{proof}

\subsection*{Rigidity}
The tensor product on $\bfK(\Prj G)$ induces one on $\lfK(\Prj G,R)$, the subcategory introduced following Lemma~\ref{le:KProj-compact}. By construction, it is compactly generated and our aim is to show that it is also rigidly-compactly generated, in the sense recalled below.

Let $\cat T$ be a compactly generated tensor triangulated category,
with tensor product $\otimes$ and unit $\one$. We assume that $\one$ is
compact. Being a compactly generated tensor triangulated category,
$\cat T$ has an \emph{internal function object}, $\fHom(-,-)$, defined
by the property that
\[
\Hom_{\cat T}(X\otimes Y,Z) \cong \Hom_{\cat T}(X,\fHom(Y,Z))
\]
for $X,Y$ and $Z$ in $\cat T$. There is a natural map
\begin{equation}
\label{eq:rigidity}
\fHom(X,\one)\otimes Y \lra \fHom(X,Y)
\end{equation}
and $X$ is \emph{rigid} if this map is an isomorphism for all
$Y$. Since $\one$ is compact, every rigid object is compact, and one
says that $\cat T$ is rigidly-compactly generated when the converse holds: compact
objects and rigid objects coincide. It is straightforward to verify that this property holds if and only if:
\begin{enumerate}[\quad\rm(1)]
\item
the subcategory of compact objects is closed under $\otimes$, and
\item
\eqref{eq:rigidity} is an isomorphism when $X,Y$ are compact.
\end{enumerate}

We get back to  the tensor triangulated category $\bfK(\Prj G)$ with product $\otimes_R$, where $G$ is finite flat group scheme over $R$.

\begin{proposition}
\label{pr:rigidity}
The tensor product $\otimes_R$ induces on $\lfK(\Prj G, R)$ a tensor-tri\-an\-gu\-lated structure that is rigidly-compactly generated.
\end{proposition}

\begin{proof}
  We claim that the tensor product on $\bfK(\Prj G)$ restricts to a  product on $\lfK(\Prj G, R)$.  Because $\lfK(\Prj G, R)$ is compactly
  generated it suffices to show that for any compact objects $X,Y$ in $\lfK(\Prj G, R)$, the complex $X\otimes_RY$ is again in $\lfK(\Prj G, R)$. For the rigidity we need to verify that  $X\otimes_RY$ is also compact and that \eqref{eq:rigidity} is an
  isomorphism. We use the identification of the compact objects in $\lfK(\Prj G, R)$ with the objects in $\bfD^b(\rmod(G,R))$; see
  Proposition~\ref{pr:KProj}.  For $M,N$ in $\bfD^b(\rmod(G,R))$ it is clear that the complex $M\lotimes_RN$ is in $\bfD^b(\rmod(G,R))$ since $M$ and $N$ are perfect complexes over $R$. Now we can identify $\fHom(M,N)=\RHom_R(M,N)$ and it remains to observe that the map
\[
\RHom_R(M,R) \lotimes_R N \lra \RHom_R(M,N)
\]
of $G$-complexes is an isomorphism, because $M$ is perfect over $R$.
\end{proof}

\subsection*{Cohomology}
With $G$ acting trivially on $R$, we consider the cohomology algebra $S\colonequals \Ext^*_G(R,R)$ of $G$. Since $G$ is a group scheme, standard arguments  yield that $S$ is a graded-commutative $R$-algebra. For any $G$-module $M$ set
\[
H^*(G,M) \colonequals \Ext_G^*(R,M)\,.
\]
This is a graded module over $S$, under cup-products; see \cite{Benson:1998c}.  Friedlander and Suslin~\cite{Friedlander/Suslin:1997a} prove that $S$ is finitely generated---equivalently, noetherian---when $R$ is a field; this result has been extended by van der Kallen \cite{vanderKallen:2023a} to cover any noetherian commutative ring $R$; see \ref{ch:vanderkallen}.

For any map $R\to R'$ we write $G_{R'}$ for the group scheme over $R'$ obtained by base change of $G$ along $R\to R'$. This induces a map of graded-commutative rings
\[
H^*(G,R)\lra H^*(G_{R'},R')\,.
\]
For any $R'$-module $M'$ one has an isomorphism of $H^*(G,R)$-modules
\[
H^*(G,M') \cong H^*(G_{R'},M')\,,
\]
where $M'$ is viewed as a $G$-module via restriction, as usual. This is compatible with the map of graded algebras $H^*(G,R)\to H^*(G_{R'},R')$.

\subsection*{Examples}
We list examples of finite flat group schemes over rings; the purpose is to indicate that there is a plethora of such things. We are grateful to Sean Cotner for sharing his expertise on this topic.

\begin{example} 
\label{ex:tate}  
For $R$ a noetherian commutative ring (satisfying mild conditions) Oort and Tate \cite{Tate/Oort:1970a} classified finite group schemes with coordinate algebra $R[x]/(x^p)$. When $R$ is a complete local ring with residue field of characteristic $p$, there is such a group scheme for any factorisation $p=ac$ in $R$. Two such group schemes are isomorphic if the pairs $(a,c)$ and $(a^\prime, c^\prime)$ differ by multiplication by a unit in $R$. 

When $R$ contains $\mathbb F_p$ the group algebra of  such a group scheme is $R[x]/(x^p-\lambda x)$ for $\lambda \in R$, and two such group schemes are isomorphic if the $\lambda$'s differ by a unit. 
\end{example}

\begin{example}
    \label{ex:Ga} 
    To construct examples of higher rank, consider an $\mathbb F_p$-algebra $R$ and the additive group scheme $\mathbb G_{a,R}$ over $R$. Its coordinate algebra is $R[x]$. Let 
\[
f\colon \mathbb G_{a,R} \lra \mathbb G_{a,R}
\]
be a map of group schemes defined by the map on coordinate algebras
\[
f^*\colon R[x] \lra R[x], \quad f^*(x) =  x^{p^n} + a_{p^n-1}x^{p^{n-1}} + \ldots + a_1x
\] 
for some $n \geq 1$, and elements $a_i$ in $R$. This is a Hopf algebra map since $f^*(x)$ is primitive by construction and hence defines a map of group schemes. Let $G$ be the scheme theoretic kernel of $f$. Then $G$  is a finite flat group scheme over $R$ with coordinate algebra $R[x]/(f^*(x))$, of rank $p^n$.

For example, take $R = k[[t]]$, with $k$ a field of positive characteristic $p$, and $f^*(x) = x^p - t^{p-1}x$.  The fibres of $G\colonequals\Ker f$ over $R$ are different group schemes. The special fibre, at $t=0$, has the coordinate algebra $k[x]/(x^p)$, so $G_k \cong \mathbb G_{a(1)}$, the first Frobenius kernel of $\mathbb G_a$. On the other hand, the generic fibre $G_{k((t))}$ has coordinate algebra $k((t))[x]/(x^p - t^{p-1}x)$, which splits into $p$ copies of $k((t))$. Hence it is the constant finite group scheme $\bbZ/p$.
\end{example} 

Another family of commutative finite flat group schemes is given by the (scheme-theoretic) $p$-torsion subgroups of supersingular abelian varieties defined over $\mathbb F_p$-algebras; see \cite[Chapter 13]{Li/Oort:1998a}. The next family is again based on $\mathbb G_a$ but consists of noncommutative group schemes.

\begin{example}
\label{ex:noncomm}
Set $R=k[t]/(t^p)$ where  $k$ is a field of characteristic $p$. Consider the cyclic group  $\bbZ/p$ with a generator $\sigma$ acting on $\mathbb G_{a,R}$ via multiplication by $(t+1)$:  
\[
\sigma \circ s \mapsto (t+1)s
\] 
for any $s \in \mathbb G_{a,R}$. On the level of the coordinate algebra $R[\mathbb G_{a}] \cong R[x]$, the $R$-algebra automorphism $\sigma$ is defined by $\sigma(x) = (t+1)x$.  This action commutes with the Frobenius homomorphism on $\mathbb G_{a,R}$ and, hence, restricts to the Frobenius kernels $\mathbb G_{a(r),R}$. We consider finite flat groups schemes over $R$ which are semi-direct products with respect to this action:  $G_r = \mathbb G_{a(r),R} \rtimes \bbZ/p$.   

Explicitly, the group algebra $RG_r$ is the smash product: $RG_r \cong R\mathbb G_{a(r)} \# R\bbZ/p$. The coproduct is inherited from that on $R\mathbb G_{a(r)} \cong R[x_1, \ldots, x_r]/(x_1^p, \ldots, x_r^p)$ and $R\bbZ/p \cong R[\sigma]/(\sigma^p-1)$ and is cocommutative. The product is twisted with the action of $\sigma$: 
\begin{align*} 
(1 \# \sigma)(x_1 \# 1) & = (t+1)x_1\#\sigma; \\
(1 \# \sigma)(x_i \# 1) & = x_i\#\sigma, \quad 1 < i \leq r.
\end{align*} 
These are noncommutative and finite flat over $R$, with $RG_r$ a free $R$-module of rank $p^{r+1}$.   The smallest one, $G_1$, cannot be obtained via base change from a group scheme defined over any integral domain. A proof of this claim was communicated to us by Sean Cotner.
\end{example}

\section{Enhancements} 
\label{se:enhancements}

Let $G$ be a finite flat group scheme over a noetherian commutative ring $R$. In this section we exhibit an $\infty$-categorical enhancement of the triangulated category $\lfK(\Prj G,R)$ and, when $G$ is a finite group, compare it to the derived $\infty$-category $\crep(G,R)$ of representations studied in \cite{Barthel:strat,Barthel:strat-regular}.  While this enhancement  is not used in the remainder of this paper, it plays an important role in related developments; see \ref{ch:highercats}. The reader not familiar with the language of $\infty$-categories may consult Lurie's writings~\cite{Lurie:2009a,Lurie:2017a}, or any of the other accounts that have since appeared on this topic.

We begin with the construction of an enhancement of $\bfK(\Prj A)$, for any ring $A$. Let $\Ch(\Prj A)$ be the category of (unbounded) chain complexes of projective $A$-modules, viewed as a differential graded category. Following \cite[Construction~1.3.1.6]{Lurie:2017a}, we define
\[
\enhK(\Prj A) \coloneqq \mathrm{N}_{\mathrm{dg}}(\Ch(\Prj A))
\]
as the differential graded nerve of $\Ch(\Prj A)$. This is a stable $\infty$-category whose homotopy category identifies with $\bfK(\Prj A)$; see \cite[\S1.3.1]{Lurie:2017a}.

With $R$ and $G$ as above, the tensor product on $\Ch(\Prj G)$ induces a symmetric monoidal structure on $\enhK(\Prj G)$. It follows as in the proof of Proposition~\ref{pr:rigidity} that the localising subcategory of $\enhK(\Prj G)$ spanned by those objects which satisfy the equivalent conditions of Lemma~\ref{le:KProj-compact} inherits a structure of symmetric monoidal stable $\infty$-category. We denote this $\infty$-category $\enhK(\Prj G,R)$; by construction, its homotopy category identifies with $\lfK(\Prj G,R)$.

Let $\enhD^b(\rmod(G,R))$ be the symmetric monoidal stable $\infty$-category underlying $\bfD^b(\rmod(G,R))$.\footnote{As in \cite{Lau:2023a}, one observes that $\bfD^b(\rmod(G,R))$ is equivalent to the derived category of perfect complexes on the quotient stack $\Spec(R) \sslash G$.} Since compact objects, generators, and rigidity are detected at the level of the homotopy categories, Propositions~\ref{pr:KProj} and~\ref{pr:rigidity} translate into the statement that $\enhK(\Prj G,R)$ is rigidly-compactly generated with full subcategory of compact objects determined by the equivalence
\[
\enhK(\Prj G,R)^c\longiso\, \enhD^b(\rmod(G,R))\,,
\]
induced by the localisation functor that inverts the quasi-isomorphisms. Equivalently, $\enhK(\Prj G,R)$ is the ind-completion of $\enhD^b(\rmod(G,R))$, i.e., there is an equivalence
\begin{equation}
\label{eq:enhrep}
\enhK(\Prj G,R) \longiso\, \Ind\enhD^b(\rmod(G,R))\,,
\end{equation}
of rigidly-compactly generated symmetric monoidal stable $\infty$-categories. Passage to the homotopy category thus provides a model for $\lfK(\Prj G,R)$. We refer to \cite[\S5.3.5]{Lurie:2009a}, \cite[Proposition~1.1.3.6]{Lurie:2017a}, and \cite[Corollary~4.8.1.14]{Lurie:2017a} for a discussion of ind-completions of $\infty$-categories and their monoidal properties. 

In the remainder of this section, we specialise to the case that $G$ is a finite group. In \cite{Barthel:strat}, the first author introduced a derived $\infty$-category of $R$-linear $G$-representation; we recall the definition.

\begin{chunk}
\label{def:repcat}
 Write $BG$ for the classifying space of the finite group $G$, viewed as an $\infty$-groupoid. The derived $\infty$-category of representations is 
\[
\crep(G,R) \coloneqq \Ind\Fun(BG,\enhD^{\mathrm{perf}}(R))\,,
\]
the ind-category of the category of local systems on $BG$ with coefficients in perfect $R$-modules. Equipped with the pointwise tensor product, $\crep(G,R)$ has the structure of a symmetric monoidal stable $\infty$-category.
\end{chunk}

The starting point for our comparison between $\crep(G,R)$ and $\enhK(\Prj G,R)$ is the observation that, by construction, $\crep(G,R)$ is rigidly-compactly generated with full subcategory of compact objects given by $\Fun(BG,\enhD^{\mathrm{perf}}(R))$. Consider first the $\infty$-category $\Fun(BG,\enhD(R))$ of local systems on $BG$ with coefficients in $\enhD(R)$, equipped with the pointwise symmetric monoidal structure. Let $\enhD(RG)$ be the derived $\infty$-category of $RG$-modules with symmetric monoidal structure coming from the coproduct on the cocommutative Hopf algebra $RG$, as above; its homotopy category is $\bfD(\Mod(RG))$. The next result is folklore, so we only sketch the proof.

\begin{lemma}\label{lem:localsystem}
Let $G$ be a finite group and $R$ a commutative ring. There is an equivalence 
$\Fun(BG,\enhD(R)) \simeq \enhD(RG)$ of symmetric monoidal stable $\infty$-categories, which restricts to a symmetric monoidal equivalence
\[
\Fun(BG,\enhD^{\mathrm{perf}}(R)) \simeq \enhD^b(\rmod(G,R))\,.
\]
\end{lemma}
\begin{proof}
The $R$-linear dual $A = \fHom(RG,R)$ of $RG$ has the structure of a dualisable commutative algebra in $\enhD(G,R)$. Induction along the unit map $R \to A$ identifies with the restriction functor $\mathrm{Res}\colon \enhD(RG) \to \enhD(R)$ and the corresponding cosimplicial descent diagram takes the form
\[
\Delta \lra \mathrm{Cat}_{\infty}, \qquad \bullet \mapsto \textstyle\prod_{G^{\times \bullet}}\enhD(R) \coloneqq (\enhD(R) \rightrightarrows \prod_{G}\enhD(R) \Rrightarrow \ldots)\,.
\]
We thus obtain a symmetric monoidal exact comparison functor to the totalisation,
\begin{equation}\label{eq:deriveddescent}
\enhD(RG) \lra \mathrm{Tot}\textstyle\prod_{G^{\times \bullet}}\enhD(R)\,.
\end{equation}
On the one hand, since $\mathrm{Res}$ is conservative, it follows from the Barr--Beck--Lurie theorem \cite[Theorem~4.7.3.5]{Lurie:2017a} that the functor \eqref{eq:deriveddescent} is an equivalence. On the other hand, the standard simplicial bar resolution of $BG$ implies that the totalisation in \eqref{eq:deriveddescent} recovers the category of local systems on $BG$ with coefficients in $\enhD(R)$,
\[
\Fun(BG,\enhD(R)) \simeq \mathrm{Tot}\textstyle\prod_{G^{\times \bullet}}\enhD(R)\,,
\]
equipped with its pointwise symmetric monoidal structure. Composing these two symmetric monoidal equivalences, we obtain the first part of the claim:
\begin{equation}\label{eq:localsystem1}
    \Fun(BG,\enhD(R)) \simeq \enhD(RG)\,. 
\end{equation}  
The equivalence of \eqref{eq:localsystem1} restricts to an equivalence on dualisable objects:
\begin{equation}\label{eq:localsystem2}
    \Fun(BG,\enhD(R))^{\mathrm{dbl}} \simeq \enhD(RG)^{\mathrm{dbl}}\,.
\end{equation}
On the one hand, since the tensor product is pointwise, the left side of \eqref{eq:localsystem2} identifies with $\Fun(BG,\enhD(R)^{\mathrm{dbl}}) \simeq \Fun(BG,\enhD(R)^{\mathrm{perf}})$. On the other hand, the conservative functor $\enhD(RG) \to \enhD(R)$ detects dualisable objects, so the right side of \eqref{eq:localsystem2} is naturally equivalent to $\enhD^b(\rmod(G,R))$ by Proposition~\ref{pr:cofinal}, as claimed.
\end{proof}

\begin{chunk}
Let $i^*R \in \Fun(BG, \enhD(R))$ be the local system arising from pushing $R$ forward along the inclusion of a chosen basepoint $i\colon \ast \to BG$. Unwinding the construction, the equivalence of Lemma~\ref{lem:localsystem} is determined by sending $i^*R$ to the regular representation $RG \in \enhD(RG)$.
\end{chunk}

\begin{proposition}
\label{prp:repmodel}
Let $G$ be a finite group and $R$ a noetherian commutative ring. The equivalence of Lemma~\ref{lem:localsystem} induces an equivalence 
\[
\crep(G,R) \simeq \enhK(\Prj G,R)
\]
of symmetric monoidal stable $\infty$-categories.
\end{proposition}

\begin{proof}
Recall that both sides of the desired equivalence are rigidly-compactly generated $\infty$-categories. Combining Lemma~\ref{lem:localsystem} with \eqref{eq:enhrep} thus yields equivalences
\begin{align*}
\crep(G,R) & = \Ind\Fun(BG,\enhD^{\mathrm{perf}}(R)) \\
& \simeq \Ind\enhD^b(\rmod(G,R)) \\
& \simeq \Ind(\enhK(\Prj G,R)^c) \\
& \simeq \enhK(\Prj G,R)
\end{align*}
of symmetric monoidal stable $\infty$-categories.
\end{proof}

\begin{chunk}
\label{ch:highercats}
One feature of the $\infty$-category $\crep(G,R)$ is that its construction extends to any topological group $G$ and any symmetric monoidal coefficient $\infty$-category, such as perfect modules over commutative ring spectra. This allows one to transport questions from modular representation theory to higher algebra and, vice versa, apply homotopical techniques to problems in representation theory; see \cite{Barthel/Castellana/Heard/Naumann/Pol}.
\end{chunk}

\section{Cohomology and fibres}
\label{se:cohomology-and-fibres}

As in Section~\ref{se:groups}, we fix a noetherian commutative ring $R$ and a finite flat group scheme $G$ over $R$.
For each map of rings $R\to k$ with $k$ a field, consider the $k$-algebra
\[
S(k)\colonequals \Ext_{G_k}^*(k,k)\,.
\]
The functor $-\lotimes_R k$ induces a map of graded $R$-algebras $S\to S(k)$ and hence a map of graded $k$-algebras
\begin{equation}
\label{eq:kappa}
\phi_k \colon S \otimes_R k \lra S(k)\,.
\end{equation}
The source and target are graded-commutative and finitely generated. Consider the induced map on spectra:
\begin{equation}
\label{eq:spec-kappa}
\Spec(\phi_k)\colon \Spec S(k)\lra \Spec (S \otimes_R k)\,.
\end{equation}
In \cite{Lau:2023a} it is proved that this map is a homeomorphism when $G$  is a finite group. The  argument, going back to \cite{Benson/Habegger:1987a},  extends to finite flat group schemes and yields the result below. The key new input is work of van der Kallen~\cite{vanderKallen:2023a}, recalled in \ref{ch:vanderkallen}.

\begin{theorem}
\label{th:kappa}
The map $\Spec(\phi_k)$ is a homeomorphism.
\end{theorem}

When the characteristic of $k$ is zero, the source and target of the map $\phi_k$ are zero in non-negative degrees, by a result of van der Kallen's recalled below~\ref{ch:vanderkallen}, so $\Spec(\phi_k)$ is trivially a homeomorphism. The real content of the result is in the case when the characteristic of $k$ is positive. We deduce it from a more general result, Theorem~\ref{th:uhom-main}; see \ref{ch:uhom-remarks}. Its statement and proof are modeled on those in Benson--Habegger~\cite{Benson/Habegger:1987a} and Lau~\cite[Section~8]{Lau:2023a}, and require some preparation.

\subsection*{Universal homeomorphisms}
A map $\varphi\colon A\to B$ of (possibly graded) commutative rings is a \emph{universal homeomorphism} if for each map $A\to A'$ of rings the map on spectra induced by the base change map $A'\to A'\otimes_AB$, is a homeomorphism:
\[
\Spec (A'\otimes_AB) \longiso \Spec A'\,.
\]
In particular $\Spec B\to \Spec A$ is a homeomorphism. A universal homeomorphism is universally bijective (in the obvious sense), and the converse holds if $\varphi$ is finite, but not in general: consider the map $\bbZ_{(p)}\to \bbF_p\times\bbQ$ for some prime number $p$.

Universal bijectivity is easier to track. Indeed, when $\Spec B\to \Spec A$ is surjective so is the map induced by base change. So the critical point is that the one-to-one property is preserved under base change. It is not hard to verify that the following conditions are equivalent:
\begin{enumerate}
\item
$\varphi$ is universally bijective;
\item
$\varphi\otimes_A k(\fp)\colon k(\fp)\to B\otimes_A k(\fp)$ is  universally bijective for $\fp \in \Spec A$;
\item
$\Spec(B\otimes_A k)$ is a point for each map $A\to k$ of rings, where $k$ is a field.
\end{enumerate}
This observation is used repeatedly in what follows, as is the following criterion to detect universal homeomorphisms; see~\cite[Lemma~10.46.7]{StacksProject} for a proof.

\begin{chunk}
\label{ch:uhom}
Let $\varphi\colon A\to B$ be a map of commutative rings and $p$ be a prime number. Assume the following conditions hold:
\begin{enumerate}
\item
$B$ is generated as an $A$-algebra by elements $b$ such that there exist an $n\ge 1$ for which 
$p^n\cdot b$ and $b^{p^n}$ are in $\mathrm{Image}(\varphi)$;
\item
each $a\in \Ker(\varphi)$ is nilpotent.
\end{enumerate}
Then $\varphi$ is a universal homeomorphism. In particular, if $A$ contains a field of positive characteristic $p$ and $\varphi$ is an $F$-isomorphism, it is a universal homeomorphism.
\end{chunk}

\begin{chunk}
\label{ch:change-of-coefficients}
Let $G$ be a finite flat group scheme over $R$. Any map  $f\colon R'\to R''$ of $R$-algebras induces a map
\[
H^*(f)\colon H^*(G,R')\otimes_{R'}R'' \lra H^*(G,R'')
\]
of graded $R''$-algebras. Consider the collection 
\[
\chu(G) \colonequals \{ \text{$f$ a map of $R$-algebras} \mid \text{$H^*(f)$ is a universal homeomorphism}\}.
\]
\end{chunk}

The following statements are easy to verify, given the characterisations of universal homeomorphisms and universal bijective maps.

\begin{lemma}
\label{le:chu}
The collection $\chu(G)$ is  closed under compositions and contains all flat maps. When $K$ is a commutative ring, $R$ a $K$-algebra, and $f$ a map of  $R$-algebras, the following conditions are equivalent:
\begin{enumerate}[\quad\rm(1)]
\item
 $f$ is in $\chu(G)$;
 \item
 $f\otimes_{K}K_\fp$ is in $\chu(G)$  for each $\fp$ in $\Spec K$.\qed
 \end{enumerate}
\end{lemma}

Here is a variation of the result above concerning universal bijectivity.

\begin{lemma}
\label{le:ubij}
Let $f\colon R'\to R''$ be a map of commutative rings and $\overline{f}_\fq$ the composite $R'\to R''\to k(\fq)$, where $\fq\in\Spec R''$. The following conditions are equivalent:
\begin{enumerate}[\quad\rm(1)]
\item
 $H^*(f)$ is universally bijective;
 \item
 $H^*(\overline{f}_\fq)$ is universally bijective for each $\fq\in\Spec R''$.\qed
 \end{enumerate}
\end{lemma} 

\begin{chunk}
\label{ch:vanderkallen}
Recall that $R$ is a noetherian commutative ring. Van der Kallen~\cite[Theorem~1]{vanderKallen:2023a} proves the following results:
\begin{enumerate}
\item
There exists an integer $d\ge 1$ with
\[
\label{eq:torsion-in-HA}
d\cdot H^{\geqslant 1}(G,-)=0 \quad\text{on $\Mod R$}\,.
\]
\item
The $R$-algebra $H^*(G,R)$ is finitely generated, and hence noetherian, and the $H^*(G,R)$-module $H^*(G,M)$ is finitely generated whenever $M$ is a finitely generated $G$-module.
\end{enumerate}
These results are crucial to all that follows. When $G$ is a finite group, one can take $d$ to be the order of the group. The statement is also obvious when $d\cdot R=0$, so the essential case is when the structure map $\bbZ\to R$ is one-to-one. 
\end{chunk}

\begin{lemma}
\label{le:uhom-nilpotent}
Any surjective map $f$ of $R$-algebras with $\Ker(f)$ nilpotent is in $\chu(G)$.
\end{lemma}

\begin{proof}
 It suffices to check that $f_\fp$ is in $\chu(G)$ for each $\fp$ in $\Spec R'$, by Lemma~\ref{le:chu}, so we may assume $R'$ is a local ring. In particular, it contains $\bbQ$ or its residue field is of positive characteristic $p$ and  $R'$ is $p$-local, that is to say, the natural map $R'\to \bbZ_{(p)}\otimes_{\bbZ} R'$ is bijective.

In the former case $H^{\geqslant 1}(G,-)=0$ on $\Mod R'$, and there is nothing left to verify; this is by \ref{ch:vanderkallen}(1). So we may assume $R'$ is $p$-local. Then, again by \ref{ch:vanderkallen}(1), one gets 
\[
p^n\cdot H^{\geqslant 1}(G,-)=0\quad\text{on $\Mod R'$ and some $n\ge 1$.}
\]
Set $I\colonequals \Ker(f)$; the hypothesis is that $I^s=0$ for some $s\ge 1$. It suffices to verify the statement when $s=2$, and to that end we verify that $H^*(f)$ satisfies the conditions in \ref{ch:uhom}. Condition (2) clearly holds as $I$ is nilpotent. The exact sequence
\[
0\lra I\lra R'\lra R''\lra 0
\]
of $R'$-modules induces an exact sequence
\[
H^*(G,R') \xra{\ H^*(f)\ }  H^*(G,R'')\xra{\ \delta\ } H^*(G,I)\,.
\]
Arguing as in the proof \cite[Lemma~5.2]{Benson/Iyengar/Krause/Pevtsova:2022b}, or \cite[Lemma 8.17]{Lau:2023a}, yields that the connecting map $\delta$ is a derivation: For any $x,y$ in $H^*(G,R'')$, one has
\[
\delta(x\cup y) = \delta(x)\cup y + (-1)^{|x|} x\cup \delta(y)\,.
\]

Fix $x\in H^{\geqslant 1}(G,R')$.  When $|x|$ and $p$ are both odd, $x^2=0$ by the graded-commutativity of $H^*(G,R')$. So we can assume $|x|$ is even or $p=2$. Either way, it follows from the fact that $\delta\colon H^*(G,R'')\to H^{*+1}(G,I)$ is a derivation that $\delta(x^i) = ix^{i-1}\delta(x)$. Since $p^n\cdot H^{\geqslant 1}(G,-)=0$, we deduce that $\delta(x^{p^n})=0$ and hence $x^{p^n}$ is in the image of $H^*(f)$. Thus condition (1) of \ref{ch:uhom} is also satisfied.
\end{proof}

A surjective map $R'\to R''$ is a \emph{complete intersection} if its kernel can be generated by a regular sequence.

\begin{lemma}
\label{le:uhom-nzd}
Any surjective complete intersection map $f\colon R'\to R''$ of $R$-algebras with $R'$ noetherian is in $\chu(G)$. 
\end{lemma}

\begin{proof}
Arguing as in the proof of Lemma~\ref{le:uhom-nilpotent} we can assume  $R'$ is local, with residue field of positive characteristic $p$ and that $p^n\cdot H^{\geqslant 1}(G,-)=0$ on $\Mod R'$. Moreover, since $\chu(G)$ is closed under compositions it suffices to consider the case $R''=R'/t$, where $t$ is not a zero-divisor in $R'$. At this point one can argue as in the proof of \cite[Lemma~8.21]{Lau:2023a} to deduce the desired statement. Here is a sketch of the argument.

For each integer $i\ge 1$ set
\[
J_i\colonequals \Ker (H^*(G,R')\xra{\ t^i\ } H^*(G,R''))\,.
\]
The ring $H^*(G,R')$ is noetherian, by van der Kallen's result~\ref{ch:vanderkallen}, so the chain of ideals $J_1\subseteq J_2\subseteq \cdots$ stabilises. Fix an integer $s$ such that $J_i=J_{i+1}$ for $i\ge s$. Since the map $R'/t^s\to R'/t$ is in $\chu(G)$, by Lemma~\ref{le:uhom-nilpotent}, it suffices to verify that the map $R'\to R'/t^s$ is in $\chu(G)$. Thus, replacing $t$ by $t^s$, we can assume $J_1=J_i$ for $i\ge 1$. The sequence of $G$-modules
\[
0\lra R'\xra{\  t\ } R'\xra{\ f\ } R'/t\to 0
\]
yields the exact sequence of $H^*(G,R')$-modules
\[
0\lra R'/t\otimes_{R'} H^*(G,R') \xra{\ H^*(f)\ } H^*(G,R'/t) \xra{\ \delta\ } J_1\lra 0\,.
\]
It suffices to verify $\delta(x^{p^n})=0$ for $x\in H^{\geqslant 1}(G,R'/t)=0$; see \ref{ch:uhom}.

Let $P$ be a projective resolution of $R'$ as a $G_{R'}$-module. As $t$ is not a zero-divisor on $R'$, the complex $P/tP$ is a projective resolution of $R'/t$ as a $G_{R'/t}$-module, and the exact sequence above is induced by the exact sequence
\[
0\lra \End_G(P) \lra \End_G(P) \lra \End_G(P/tP)\lra 0\,.
\]
The product on $H^*(G,R')$ and $H^*(G,R'/t)$ is induced by the composition product on the differential graded algebras $\End_G(P)$ and $\End_G(P/tP)$, respectively, and the map $H^*(f)$ is induced by the map $\End_G(P)\to \End_G(P/tP)$ of dg algebras.
Fix a cycle $z$ in $\End_G(P/tP)$ and a pre-image  $\tilde{z}$ in $\End_G(P)$; then $\delta([z]) = [d(\tilde{z})/t]$. The graded-commutativity of the ring $H^*(G,R)$ and the fact that $p^n H^{\geqslant 1}(G,R')=0$ yields $\delta([z]^{p^n})= [tw]$ for some $w\in\End_G(P)$; see \cite[Lemma~8.21]{Lau:2023a} for details. Since $td(w) = d(tw) =0$ in $\End_G(P)$ and $t$ is not a zero-divisor,  $d(w)=0$. Then
\[
t^2[w] = t[tw] = t \delta([z]^{p^n}) =0
\]
since $t\cdot \End_G(P/tP)=0$. So $[w]$ is in $J_2$. Since $J_2=J_1$ we deduce that $t[w]=0$, that is to say, that $\delta([z]^{p^n})=0$, as desired.
\end{proof}

A map $R\to R'$ of commutative rings is said to be \emph{essentially of finite type} if $R'$ is a localisation of a finitely generated $R$-algebra. 

\begin{theorem}
\label{th:uhom-main} 
Let $R$ be a noetherian commutative ring, $G$ a finite flat group scheme over $R$, and $f\colon R\to R'$ a map of commutative rings. The following statements hold.
\begin{enumerate}[\quad\rm(1)]
\item
$H^*(f)$ is universally bijective; 
\item
$H^*(f)$ is a universal homeomorphism if $f$ is essentially of finite type.
\end{enumerate}
\end{theorem}

\begin{proof}

(1) Given Lemma~\ref{le:ubij} it suffices to consider the case when $R'=k$, a field. We claim that then the map $H^*(f)$ is even a universal homeomorphism. To see this, set $\fp\colonequals \Ker(f)$, a prime ideal in $R$. The map $f\colon R\to k$ factors as 
\[
R\lra R_\fp \lra k(\fp)\lra k\,,
\]
where the first two maps are the natural ones. Since both maps $R\to R_\fp$ and  $k(\fp)\to k$ are flat and hence in $\chu(G)$, it suffices to verify that the map $R_\fp \to k(\fp)$ is in $\chu(G)$. We can thus assume $R$ is local and $f\colon R\to k$ is the quotient map, with $k$ the residue field of $R$.

We induct on $\dim R$ to deduce that $f$ is in $\chu(G)$, and in particular universally bijective. The base case $\dim R=0$ is covered by Lemma~\ref{le:uhom-nilpotent} for then the ideal $\Ker(f)$ is nilpotent.

Suppose $\dim R\ge 1$ and set $I\colonequals \sqrt{(0)}$ the nilradical of $R$; it is a nilpotent ideal, since $R$ is noetherian. Thus the map $R\to R/I$ is in $\chu(G)$. Since the map $R\to k$ factors as $R\to R/I\to k$, it suffices to verify that $R/I\to k$ is in $\chu(G)$. So we can replace $R$ by $R/I$ and assume it is reduced. Then, since $\dim R\ge 1$, the maximal ideal of $R$ is not an associated prime and hence contains an element $t$ that is not a zero-divisor. The map $R\to k$ factors as 
\[
R\lra R/tR\lra k\,.
\]
The map on the left is in $\chu(G)$, by Lemma~\ref{le:uhom-nzd}, and the one on the right is in $\chu(G)$ by the induction hypothesis, since
$\dim(R/tR)=\dim R-1$. We conclude that $f$ is in $\chu(G)$.

(2) Since $f$ is essentially of finite type,  it factors as $R\to R''\to R'$ where $R\to R''$ is the localisation of a finitely generated polynomial ring over $R$, and $R''\to R'$ is surjective. Since $R\to R''$ is flat it is in $\chu(G)$. We may thus replace $R$ by $R''$ and suppose $f$ is surjective, and hence a finite map. It follows that the map $H^*(f)$ is also finite, by van der Kallen's theorem~\ref{ch:vanderkallen}(2), and hence the induced map $\Spec H^*(f)$ is universally closed. Since $H^*(f)$ is universally bijective, by part (1), we conclude that it is a universal homeomorphism.
\end{proof}

\begin{chunk}
\label{ch:uhom-remarks}
Theorem~\ref{th:kappa} is the special case of Theorem~\ref{th:uhom-main} where $R'$ is a field. On the other hand, as can be seen from the proof, the latter result is proved by (a simple) reduction to the case where $R'$ is a field.
\end{chunk}

\section{Local cohomology and support}
\label{se:lch}

This section is a reminder about local cohomology functors and a general local-to-global principle for  triangulated categories; for details, see \cite{Benson/Iyengar/Krause:2008a, Benson/Iyengar/Krause:2011a}. 

\subsection*{Localisation and local cohomology}

Let $\cat T$ be a triangulated category with suspension $\Sigma$. Given
objects $X$ and $Y$ in $\cat T$, consider the graded abelian groups
\[ \Hom_{\cat T}^*(X,Y)=\bigoplus_{i\in\bbZ}\Hom_{\cat T}(X,\Sigma^i Y) \quad
\text{and}\quad
\End_{\cat T}^{*}(X)= \Hom_{\cat T}^{*}(X,X)\,.
\] Composition makes $\End_{\cat T}^{*}(X)$ a graded ring and
$\Hom_{\cat T}^{*}(X,Y)$ a left-$\End_{\cat T}^{*}(Y)$
right-$\End_{\cat T}^{*}(X)$ module.

Let $S$ be a graded-commutative noetherian ring.  In what follows we
will only be concerned with homogeneous elements and ideals in $S$. In
this spirit, `localisation' will mean homogeneous localisation, and
$\Spec S$ will denote the set of homogeneous prime ideals in $S$.

We say that a triangulated category $\cat T$ is \emph{$S$-linear} if for
each $X$ in $\cat T$ there is a homomorphism of graded rings
$\phi_X\colon S\to \End_{\cat T}^{*}(X)$ such that the induced left and
right actions of $S$ on $\Hom_{\cat T}^{*}(X,Y)$ are compatible in the
following sense: For any $r\in S$ and $\alpha\in\Hom^*_\cat T(X,Y)$, one
has
\[ \phi_Y(r)\alpha=(-1)^{|r||\alpha|}\alpha\phi_X(r)\,.
\]

An exact functor $F\colon \cat T\to\cat U$ between $S$-linear triangulated
categories is \emph{$S$-linear} if the induced map
\[ \Hom_{\cat T}^*(X,Y)\lra \Hom_{\cat U}^*(FX,FY)
\] of graded abelian groups is $S$-linear for all objects $X,Y$ in
$\cat T$.

In what follows, we fix a compactly generated $S$-linear triangulated
category $\cat T$ and write $\cat T^c$ for its full subcategory of compact
objects.

Let $V\subseteq \Spec S$ be a specialisation closed subset. An
$S$-module $M$ is \emph{$V$-torsion} if $M_\fp=0$ for all $\fp$ in
$\Spec S\setminus V$.  Analogously, an object $X$ in $\cat T$ is
\emph{$V$-torsion} if the $S$-module $\Hom_\cat T^*(C,X)$ is
$V$-torsion for all $C\in\cat T^c$. For an ideal $I\subseteq S$ we use
the more common term \emph{$I$-torsion} instead of $V(I)$-torsion.
The full subcategory of $V$-torsion objects
\[ 
\gam_{V}\cat T\coloneqq\{X\in\cat T\mid X \text{ is $V$-torsion} \}
\] 
is localising and the inclusion $\gam_{V}\cat T\subseteq \cat T$
admits a right adjoint, denoted $\gam_{V}$. Then each object
$X\in\cat T$ fits into a functorial exact triangle
\[
\gam_V X\lra X\lra L_V X\lra
\]
and $L_V$ denotes the corresponding localisation functor.

Fix a $\fp$ in $\Spec S$.  An $S$-module $M$ is \emph{$\fp$-local} if
the localisation map $M\to M_\fp$ is invertible, and an object $X$ in
$\cat T$ is \emph{$\fp$-local} if the $S$-module $\Hom_\cat T^*(C,X)$ is
$\fp$-local for all $C\in\cat T^c$.  Consider the full subcategory of
$\cat T$ of $\fp$-local objects
\[ 
\cat T_\fp\coloneqq\{X\in\cat T\mid X \text{ is $\fp$-local}\}
\] 
and the full subcategory of $\fp$-local and $\fp$-torsion objects
\[ 
\gam_\fp\cat T\coloneqq\{X\in\cat T\mid X \text{ is $\fp$-local and $\fp$-torsion} \}.
\] 
Note that $\gam_\fp\cat T\subseteq\cat T_\fp\subseteq\cat T$ are localising subcategories.  The inclusion $\cat T_\fp\to\cat T$ admits a left adjoint $X\mapsto X_\fp$, which identifies with the localisation functor with respect to the specialisation closed set $\{\fq\in\Spec S\mid \fq\not\subseteq \fp\}$, while the inclusion $\gam_\fp\cat T\to\cat T_\fp$ admits the right adjoint $X\mapsto \gam_{V(\fp)}X$. We denote  $\gam_\fp\colon\cat T\to\gam_\fp\cat T$ the composite of those adjoints; it is the \emph{local cohomology functor} with respect to $\fp$.

\subsection*{Support}

Let $\cat T$ and $S$ be as above. The \emph{support} of an object
$X\in\cat T$ is 
 \[
\supp_S X\colonequals \{\fp\in\Spec S\mid \gam_{\fp}X\ne 0\}\,.
\] 

The next result is an important input in the Section~\ref{se:stratification}.

\begin{lemma}
\label{le:supp-adjoint}
Let $\cat T$ and $\cat U$ be compactly generated $S$-linear
triangulated categories. Let $E\colon \cat T\to \cat U$ be an
$S$-linear exact functor that preserves arbitrary coproducts and
compact objects. With $F$ a right adjoint of $E$, for each
$X\in \cat T$ one has
\[
\supp_S E(X) = \supp_S FE(X)\,.
\]
\end{lemma}

\begin{proof}

Since $E$ preserves compact objects, its right adjoint $F$ preserves coproducts. It thus follows from \cite[Proposition~7.2]{Benson/Iyengar/Krause:2012b} that the local cohomology functor $\gam_{\fp}$ commutes with $E$ and $F$, for  $\fp$ in $\Spec S$. In particular $\gam_\fp E(X)$ is isomorphic to $E(\gam_\fp X)$, and hence in the essential image of $E$. As $F$ is a right adjoint of $E$ it is conservative on the essential image of $E$, so one gets the first equivalence below: 
\[
\gam_\fp E(X)=0 \iff F(\gam_\fp E(X))=0 \iff  \gam_\fp FE(X) =0\,.
\]
The second equivalence holds by the observation above. This is as desired.
\end{proof}

\subsection*{Tensor triangulated categories}
From now on $\cat T$ is a rigidly-compactly generated triangulated category, with product $\otimes$ and unit $\one = \one_{\cat T}$. We assume that there is a central action of $S$ via a ring homomorphism $S\to\Hom^*_{\cat T}(\one,\one)$. By \cite[Theorem~8.2]{Benson/Iyengar/Krause:2008a}, this implies that for each specialisation closed subset $V$ of $\Spec S$ there are natural isomorphisms
\[
\gam_VX \cong X\otimes \gam_V\one \quad\text{and}\quad L_VX \cong X\otimes L_V\one\,.
\]
In particular $\gam_{\fp} X\cong X\otimes \gam_{\fp}\one$  for each $\fp\in\Spec S$.

For any class of objects $\cat X$ in $\cat T$ we write $\Loc^\otimes(\cat X)$ for the smallest localising tensor ideal of $\cat T$ containing $\cat X$. The following local-to-global result for tensor triangulated categories is
\cite[Theorem~3.6]{Benson/Iyengar/Krause:2011b}.

\begin{theorem}
\label{th:local-to-global}
\pushQED{\qed}   
For each object $X$ in $\cat T$ we have
\[ 
\Loc^\otimes(X)=\Loc^\otimes(\{\gam_{\fp}X\mid \fp\in\Spec
  S\})\,.\qedhere
  \]
\end{theorem}

\section{Building and fibres}
\label{se:building}

Let $R$ be a noetherian commutative ring and $G$ a finite flat group scheme over $R$. To lighten the notation we use the  abbreviations
\[
\Rep(G,R) \colonequals \lfK(\Prj G,R)\quad \text{and}\quad \rep(G,R) \colonequals \bfD^b(\rmod(G,R))\,.
\]
 One has $\Rep(G,R)^c = \rep(G,R)$, by Proposition~\ref{pr:KProj}. Moreover, when $R$ is a field, one has $\Rep(G,R) = \bfK(\Prj G)$, by Proposition~\ref{pr:KProj-regular}.

Let $k(\fp)$ be the residue field of $R$ at $\fp\in \Spec R$, and for any $G$-complex $X$ set
\[
\fibre X{\fp} \colonequals X\otimes_R k(\fp)\,.
\]
The assignment $X\mapsto \fibre X{\fp}$ induces an exact functor
\[
\pi_{\fp}\colon \Rep(G, R) \lra \Rep(G_{k(\fp)}, k(\fp)) =\bfK(\Prj G_{k(\fp)})\,.
\]
Here is the main result in this section.

\begin{theorem}
\label{th:lg-kproj-ha}
Let $R$ be a noetherian commutative ring, $G$ a finite flat group scheme over $R$, and fix $X,Y$ in $\Rep(G,R)$. The following statements are equivalent.
\begin{enumerate}[\quad\rm(1)]
\item
 $X\in  \Loc^{\otimes}(Y)$ in $\Rep(G,R)$;
\item
 $\fibre X\fp\in \Loc^{\otimes}(\fibre Y{\fp})$ in $\Rep(G_{k(\fp)}, k(\fp))$ for each $\fp\in \Spec R$.
\end{enumerate}
\end{theorem}

The proof is given towards the end of the section.

\subsection*{Reduction to the local case}

Fix $\fp\in \Spec R$.  For any $G$-complex $X$ set
\[
\lambda(X)\colonequals X\otimes_R R_\fp
\] 
viewed as a complex of $G_\fp$-modules, where $G_\fp$ is shorthand for the group scheme $G_{R_\fp}$ over $R_\fp$. As $\lambda$ preserves coproducts and compact objects it has a right adjoint, by Brown representability, and we get an adjoint pair of coproduct-preserving functors
\begin{equation*}
\label{eq:localisation}
\begin{tikzcd}
	\Rep(G,R) \arrow[yshift=.5ex]{r}{\lambda} 
    	& \Rep(G_\fp, R_\fp)\,. \arrow[yshift=-.5ex]{l}{\lambda_r}
\end{tikzcd}
\end{equation*}

\begin{lemma}
\label{le:reduction}
The composite  $\lambda_r\lambda$ is naturally isomorphic to  the functor $X\mapsto X_\fp$. Thus $\lambda_r$ identifies $\Rep( G_\fp, R_\fp)$ with the subcategory of $\fp$-local objects in $\Rep(G, R)$.
\end{lemma}

\begin{proof}
  The functor $\lambda$ identifies with the localisation functor
  \[
  -\otimes_R R_\fp\colon \rep(G,R) \lra \rep(G_\fp,R_\fp)
  \]
  when restricted to compact objects; see Proposition~\ref{pr:KProj}. Thus for compact objects $X,Y$ in
  $\Rep(G, R)$ the unit $Y\to \lambda_r\lambda(Y)$ induces an
  isomorphism
  \[
  \Hom_{\bfK(G)}(X,Y)_\fp \cong\Hom_{\bfK(G_\fp)}(\lambda X,\lambda Y)\cong
    \Hom_{\bfK(G)}( X,\lambda_r\lambda (Y))\,.
   \] 
   This yields the first assertion since $\lambda$ and $\lambda_r$ preserve coproducts. The
  second assertion is an immediate consequence.
\end{proof}

\subsection*{The local case}

For the proof of Theorem~\ref{th:lg-kproj-ha} the key is to analyse the situation where $(R,\fm,k)$ is a noetherian commutative local
ring, with maximal ideal $\fm$ and residue field $k$. One has an adjoint pair $(\pi,\pi_r)$ of functors
\begin{equation}
\label{eq:adjunction}
\begin{tikzcd}
    \Rep(G, R) \arrow[yshift=.5ex]{r}{\pi} 
    	& \Rep(G_k,k) \arrow[yshift=-.5ex]{l}{\pi_r}
\end{tikzcd}
\end{equation}
where $\pi(X)\colonequals X\otimes_Rk$. The right adjoint exists because $\pi$ preserves coproducts; since $\pi$ preserves compact objects, it follows that $\pi_r$ also preserves coproducts.

Observe that $\pi$ is a tensor functor. Because of rigidity this yields the following \emph{projection formula}. For $X$ in $\Rep(G, R)$ and $Y$ in $\Rep(G_k,k)$, one has
\[
X\otimes_R \pi_r (Y) \cong \pi_r(\pi(X)\otimes_k Y)\,.
\]
One can verify this directly, but this is a general fact about tensor
functors between rigidly-compactly generated tensor triangulated categories; see
\cite[Theorem~1.3]{Balmer/DellAmbrogio/Sanders:2016a}.

\begin{lemma}
\label{le:lg-local-case}
For any $X$ in $\Rep(G, R)$ there is an equality
\[
\Loc^{\otimes}(\pi_r\pi(X)) = \Loc^{\otimes}(\gam_{\fm}X)\,.
\]
\end{lemma}

\begin{proof}
Let $\bfi_{G_k} k$ denote the injective resolution of the trivial $G_k$-module $k$, which is the unit of the tensor triangulated
category $\KProj{G_k}$. We use that $\pi$ is a tensor functor, and in particular that
\[
\pi(\bfi R) = (\bfi R)\otimes_R k \cong \bfi_{G_k}k\,.
\]
Because $\Rep(G, R)$ is rigid with unit $\bfi R$, one has 
\[
\gam_{\fm}X \cong X\otimes_R \gam_{\fm}(\bfi R)\,.
\]
The projection formula for the pair $(\pi,\pi_r)$ and the fact that $\bfi_{G_k}k$ is the unit of $\Rep(G_k,k)$ yields isomorphisms
\[
 \pi_r\pi(X)  \cong \pi_r(\pi(X)\otimes_k \bfi_{G_k}k) \cong  X\otimes_R \pi_r(\bfi_{G_k}k)\,.
\]
Thus it suffices to verify that
\[
\Loc^{\otimes}(\pi_r(\bfi_{G_k}k)) = \Loc^{\otimes}(\gam_{\fm}(\bfi R))\,.
\]

One can verify easily that $\pi_r(\bfi_{G_k}k)$ is $\fm$-torsion, so one has
\[
\pi_r(\bfi_{G_k}k) \cong \gam_{\fm} \pi_r(\bfi_{G_k}k) \cong \pi_r(\bfi_{G_k}k)\otimes_R \gam_{\fm}(\bfi R)\,. 
\]
Thus $\pi_r(\bfi_{G_k}k)$ is in the tensor ideal localising subcategory generated by $\gam_{\fm}(\bfi R)$. 

To complete the proof, it suffices to check that  $\gam_{\fm}(\bfi R)$ is in the localising (and in particular the tensor ideal localising)
subcategory generated by $\pi_r(\bfi_{G_k}k)$. Indeed
\[
\gam_{\fm}(\bfi R)\in\Loc(\kos{(\bfi R)}{\fm})
\] 
by \cite[Proposition~2.11]{Benson/Iyengar/Krause:2011a}), so it suffices to verify that
\[
\kos{(\bfi R)}{\fm}\in\Thick(\pi_r(\bfi_{G_k}k))\,.
\] 
The key is that this can be checked in the derived category of $G$. To that end consider the diagram of pairs of adjoint functors
\[
\begin{tikzcd}[column sep=large]
    \Rep(G, R) \arrow[xshift=-.5ex]{d}[swap]{\pi} \arrow[yshift=.5ex]{r}{\bfq} 
    	& \dcat G \arrow[xshift=-.5ex]{d}[swap]{-\lotimes_Rk}  \arrow[yshift=-.5ex]{l}{\bfi} \\
	\Rep(G_k,k) \arrow[xshift=.5ex]{u}[swap]{\pi_r} \arrow[yshift=.5ex]{r}{\bfq_{G_k}} 
    	& \dcat{G_k} \arrow[xshift=.5ex]{u}[swap]{\mathrm{restriction}} \arrow[yshift=-.5ex]{l}{\bfi_{G_k}}\,. 
\end{tikzcd}
\]
It is straightforward to see that the composition of left adjoints from the top left to the bottom right commute:
\[
\bfq_{G_k}\circ \pi \cong \bfq(-)\lotimes_R k\,.
\]
It follows that the composition of right adjoints going the other way commute as well, which allows us to deduce that
\[
\pi_r(\bfi_{G_k}k) \cong \bfi k\,.
\]
Moreover, as $\bfi$ is an $R$-linear functor, one has
\[
\kos{(\bfi R)}{\fm} \cong \bfi (\kos R{\fm})\,.
\]
In $\dcat R$ the object $k$ finitely builds $\kos R{\fm}$ since the
cohomology of $\kos R{\fm}$ has finite length. The action of $G$ on $R$ and $k$ factors through the
augmentation $RG\to R$. Thus the object $k$ finitely builds $\kos R{\fm}$ in $\dcat G$. We conclude that
\[
\kos{(\bfi R)}{\fm} \cong \bfi (\kos R{\fm})\in\Thick(\bfi  k)=\Thick(\pi_r(\bfi_{G_k}k))\,.
\]
This completes the proof that $\pi_r\pi(X)$ and $\gam_{\fm}X$ generate the same tensor ideal localising subcategories.
\end{proof}

\begin{proof}[Proof of Theorem~\ref{th:lg-kproj-ha}]
(1) $\Rightarrow$ (2) This implication is straightforward to verify once one observes that the functor $X\mapsto X\otimes_R k(\fp)$  respects coproducts, and the tensor product: For  $X,Y$ in $\Rep(G, R)$ there is a natural isomorphism
\[
(X\otimes_RY)_{k(\fp)} \cong X_{k(\fp)} \otimes_{k(\fp)} Y_{k(\fp)} \quad\text{in}\quad \Rep(G_{k(\fp)},k(\fp))\,.
\]

(2) $\Rightarrow$ (1)  Assume first that $R$ is local and consider the adjoint pair \eqref{eq:adjunction}. Suppose that
$\pi(X)\in\Loc^\otimes(\pi(Y))$. It follows that $\pi_r\pi(X)$ is in $\Loc^\otimes(Y)$. Indeed, one has $\Loc^\otimes(\pi(Y)) = \Loc(\rep(G_{k(\fp)},k(\fp)) \otimes \pi(Y))$, so the projection formula implies 
\begin{align*}
\pi_r\pi(X) & \in \Loc(\pi_r(\rep(G_{k(\fp)},k(\fp)) \otimes \pi(Y))) \\
& = \Loc(\pi_r(\rep(G_{k(\fp)},k(\fp))) \otimes Y) \\
& \subseteq \Loc^{\otimes}(Y)\,.
\end{align*}
Then Lemma~\ref{le:lg-local-case} implies  $\gam_\fm X\in\Loc^\otimes(Y)$. For general $R$ and $\fp$ in $\Spec R$ we have
$X_{k(p)}=(X_\fp)_{k(p)}$ and therefore $X_{k(\fp)}\in\Loc^\otimes(Y_{k(\fp)})$ implies
\[
\gam_\fp
  X=\gam_\fp(X_\fp)\in\Loc^\otimes(Y_\fp)\subseteq\Loc^\otimes(Y)\,.
\]
Here we use Lemma~\ref{le:reduction} to reduce to the local case, where $\fp$ identifies with the maximal ideal of $R_\fp$. 
Using Theorem~\ref{th:local-to-global} we conclude that $X$ is in $\Loc^\otimes(Y)$.
\end{proof}

\section{Stratification}
\label{se:stratification}
Let $S$ be a graded-commutative noetherian ring and $\cat T$ a rigidly-compactly generated tensor triangulated category. Assume that $S$ acts on $\cat T$ via a map of graded rings $S\to \Hom^*_{\cat T}(\one,\one)$.  We say that the tensor triangulated category $\cat T$ is
\emph{stratified} by the action of $S$ if for any pair of objects $X,Y$ in $\cat T$
there are equivalences
\begin{align*}
\supp_SX\subseteq \supp_SY 
	&\iff \Loc^{\otimes}(X) \subseteq \Loc^{\otimes}(Y) \\
 	&\iff X\in \Loc^{\otimes}(Y)\,.
\end{align*}
Lemma~4.12 in \cite{Benson/Iyengar/Krause/Pevtsova:2022b} reconciles this notion with the one in \cite[Section 7.2]{Benson/Iyengar/Krause:2011a}.

Let $R$ be a noetherian commutative ring and $G$ a finite flat group scheme over $R$. We consider the cohomology algebra $S\colonequals \Ext^*_G(R,R)$ of $G$ and the cohomology algebras of the fibres of the $R$-algebra $RG$. For each $\fp$
in $\Spec R$ we set
\[
S(\fp)\colonequals \Ext_{\fibre G\fp}^*(k(\fp),k(\fp))\,.
\]
The functor $-\lotimes_R k(\fp)$ induces a map of graded $R$-algebras $S\to S(\fp)$ and this induces a natural map of graded $k(\fp)$-algebras
\[ 
\kappa_\fp \colon S \otimes_R k(\fp) \lra S(\fp)\,.
\]
The result below contains the first part of Theorem~\ref{ithm:main}. Recall that $\Rep(G,R)$ is an abbreviation for $\lfK(\Prj G,R)$.

\begin{theorem}
\label{th:stratification}
Let $R$ be a noetherian commutative ring, $G$ a finite flat group scheme over $R$. The tensor triangulated category $\Rep(G, R)$ is stratified by the action of $S\colonequals\Ext^*_G(R,R)$,  with support equal to $\Spec S$. In particular one gets a computation of the Balmer spectrum of the derived category of lattices:
\[
\Spc(\rep(G,R)) \longiso \Spec H^*(G,R)\,.
\]
\end{theorem}
As noted before, Lau~\cite{Lau:2023a} computes the Balmer spectrum of $\rep(G,R)$ when $G$ is a finite group, and where the $G$-action on $R$ is even allowed to be nontrivial.

The proof of the theorem above requires some preparation.

\subsection*{Support and fibres}
As we will explain in the proof of Theorem~\ref{th:stratification} below, Theorem~\ref{th:lg-kproj-ha} combined with the main theorem of \cite{Benson/Iyengar/Krause/Pevtsova:2018a} yields a stratification of $\Rep(G, R)$ in terms of subsets of the space
\[
\bigsqcup_{\fp\in\Spec R} \Spec S(\fp)\,,
\]
where to each $X$ in $\Rep(G, R)$ we associate the subset 
\[
\bigsqcup_{\fp\in\Spec R} \supp_{S(\fp)} \fibre X{\fp}\,.
\]
The task is to relate this to $\supp_S X$, viewed as a subset of $\Spec S$.

To that end consider the structure map $\eta\colon R\to S$, which induces a map
\[
\eta^{a}\colon \Spec S \lra \Spec R\,.
\]
The fibre of this map over $\fp$ is 
\[
(\eta^{a})^{-1}(\fp) = \Spec(S\otimes_R k(\fp))\,,
\]
which we identify with a subset of $\Spec S$ in the usual way. In the
following we use the map $\kappa_{\fp}\colon S\otimes_R k(\fp)\to S(\fp)$ which
induces the map
\begin{equation*}
\label{eq:fibre-map}
\kappa^a_{\fp}\colon \Spec S(\fp) \lra \Spec (S\otimes_R k(\fp))\subseteq \Spec S\,.
\end{equation*}

Theorem~\ref{th:kappa} yields that this map is a bijection; this fact is critical to all that follows.
The result below tracks the behavior of supports as we pass to the fibres. 

\begin{lemma}
\label{le:support-fibre}
For $\fp$ in $\Spec R$ and $X$ in $\Rep(G, R)$  there is an equality
\[
\kappa^a_{\fp}(\supp_{S(\fp)}\fibre X{\fp}) = \supp_SX \cap (\eta^{a})^{-1}(\fp)\,.
\]
\end{lemma}

\begin{proof}
  By Lemma~\ref{le:reduction}, we can  assume $(R,\fm,k)$ is a local ring and
  $\fp=\fm$, the maximal ideal of $R$. Set
  $S_k\colonequals S\otimes_Rk =S/\fm S$. Via the map
  $\kappa_\fm\colon S_k\to S(\fm)$ the $S(\fm)$-action on
  $\KProj{G_k}$ induces an $S_k$-action. Applying
  \cite[Corollary~7.8(1)]{Benson/Iyengar/Krause:2012b} to the identity
  functor on $\KProj {G_k}$ yields the first equality below
\begin{align*}
  \kappa^a_{\fm}(\supp_{S(\fm)}X_k)
  &= \supp_{S_k} X_k \\
  &= \supp_{S} X_k \\
	&=\supp_{S}(\pi_r\pi X) \\
	&=\supp_{S}(\gam_\fm X) \\
	&= \supp_{S}(X)\cap V(\fm S) \\
	&=\supp_{S}(X)\cap (\eta^a)^{-1}(\fm)\,.
\end{align*}
The same observation applied to the map $S\to S_k$ yields the second
equality.  Now consider the adjunction \eqref{eq:adjunction}.  Then
Lemma~\ref{le:supp-adjoint} gives the third equality, and the next one
follows from Lemma~\ref{le:lg-local-case}. Now observe that
$\gam_\fm=\gam_{V(\fm)}=\gam_{V(\fm S)}$ by
\cite[Proposition~7.5]{Benson/Iyengar/Krause:2012b}. Then
\cite[Theorem~5.6]{Benson/Iyengar/Krause:2008a} yields the fifth
equality, and the last one is clear since
$V(\fm S)= (\eta^a)^{-1}(\fm)$.
\end{proof}

\begin{proof}[Proof of Theorem~\ref{th:stratification}]
The main task is to verify that when $X,Y$ are objects in $\Rep(G, R)$ with $\supp_SX\subseteq \supp_SY$, one has $X\in \Loc^{\otimes}(Y)$. By Theorem~\ref{th:kappa}, for each $\fp$ in $\Spec R$ the map $\kappa_\fp^a$ is a bijection, so  Lemma~\ref{le:support-fibre} yields an  inclusion
\[
\supp_{S(\fp)}\fibre X{\fp}\subseteq \supp_{S(\fp)}\fibre Y{\fp}\,.
\]
Since $\fibre G{\fp}$ is a finite flat group scheme algebra over $k(\fp)$, the triangulated category $\KProj{G_{k(\fp)}}$
is stratified by the action of it cohomology algebra, $S(\fp)$; this
is the main result of
\cite{Benson/Iyengar/Krause/Pevtsova:2018a}. Thus the inclusion above
implies
\[
\fibre X{\fp}\in \Loc^{\otimes}(\fibre Y{\fp})\,.
\]
This holds for each $\fp$ in $\Spec R$, so we can apply
Theorem~\ref{th:lg-kproj-ha} to deduce that $X$ is in
$\Loc^{\otimes}(Y)$ as desired.

It remains to observe that $\supp_S \one= \Spec S$. This follows from
Lemma~\ref{le:support-fibre}, for $\fibre {\one}{\fp}$ is the unit of
$\KProj {G_{k(\fp)}}$ and its support is $\Spec S(\fp)$.

A standard argument---see,~\cite[Theorem 6.1]{Benson/Iyengar/Krause:2011a}---yields the last part of the statement, concerning the Balmer spectrum of $\rep(G,R)$, for this identifies with the category of compact objects in $\Rep(G,R)$.
\end{proof}

\section{Costratification and other applications}
\label{se:costratification}

The next result establishes costratification for $\Rep(G, R)$ in the sense of \cite{Benson/Iyengar/Krause:2012b}. Together with Theorem~\ref{th:stratification}, this completes the proof of Theorem~\ref{ithm:main}.

\begin{theorem}
\label{th:costratification}
Let $R$ be a noetherian commutative ring and $G$ a finite flat group scheme over $R$. The tensor triangulated category $\Rep(G, R)$ is costratified by the action of $\Ext^*_{G}(R,R)$.
\end{theorem}

\begin{proof}
We again set $S \colonequals \Ext^*_{G}(R,R)$. We verify the conditions of the bootstrap theorem \cite[Theorem~17.2]{Barthel/Castellana/Heard/Sanders}, that is to say that the collection of functors 
\[
\pi_{\fp}\colon \Rep(G,R)  \longrightarrow \Rep(G_{k(\fp)}, k(\fp)), 
\]
indexed on $\fp \in \Spec R$, has the following properties:
    \begin{enumerate}
        \item the images of the maps $(\varphi_{\mathfrak p} \coloneqq \Spc(\pi_{\fp}))_{\fp \in \Spec R}$ cover $\Spc(\rep(G,R))$;
        \item $\Rep(G,R)$ is stratified by the action of $S$; and
        \item the categories $\Rep(G_{k(\fp)},k(\fp))$ are costratified for all $\fp \in \Spec R$.
    \end{enumerate}
Then $\Rep(G,R)$ is costratified over $\Spc(\rep(G,R))$, and hence is also costratified by the action of $S$ in light of the homeomorphism $\Spc(\rep(G,R)) \cong \Spec S$.

Condition (2) holds by Theorem~\ref{th:stratification}, while the costratification of $\Rep(G_{k(\fp)},k(\fp))$ has been established in \cite{Benson/Iyengar/Krause:2012b}. Appealing to \cite[Theorem 1.3]{Barthel/Castellana/Heard/Sanders:surj}, Condition (1) follows from Proposition~\ref{pr:conserve}. Here is a direct argument for (1). Indeed, consider the commutative diagram
\[
\resizebox{\columnwidth}{!}{$\displaystyle
\begin{tikzcd}[ampersand replacement=\&]
	{\Spc(\rep(G_{k(\fp)},k(\fp)))} \&\& {\Spc(\rep(G,R))} \\
	{\Spec \Ext_{\fibre    G\fp}^*(k(\fp),k(\fp))} \& {\Spec S \otimes_R k(\fp)} \& \Spec S \\
	\& {\Spec k(\fp)} \& {\Spec R,} 
	\arrow["{\varphi_{\mathfrak p}}", from=1-1, to=1-3]
	\arrow["\rho"', "\sim" labl, from=1-1, to=2-1]
	\arrow["{\Spec \kappa_\fp}"', "\sim", from=2-1, to=2-2]
	\arrow[from=2-2, to=2-3]
	\arrow["\rho"', "\sim" labl, from=1-3, to=2-3]
	\arrow[from=2-2, to=3-2]
	\arrow[from=3-2, to=3-3]
	\arrow[from=2-3, to=3-3]
\end{tikzcd}$}
\]
in which the maps $\rho$ are Balmer's comparison maps from \cite{Balmer:2010b}, the bottom right vertical map is induced by the degree $0$ inclusion $R \to S$, and the horizontal maps are the ones induced by base change. By construction, the right lower square is a pullback in topological spaces. Theorem~\ref{th:stratification} shows that the maps $\rho$ are homeomorphisms; moreover, $\Spec \kappa_\fp$ is a homeomorphism, by Theorem~\ref{th:kappa}. For $\mathfrak p$ varying through the points of $\Spec R$, this implies that the images of $\varphi_{\mathfrak p}$ jointly cover $\Spc(\rep(G,R))$, as desired. 
\end{proof}

\begin{chunk}
It is possible to give a proof of this theorem by transcribing the argument from \cite{Barthel/Castellana/Heard/Sanders} to the theory of costratification as developed in \cite{Benson/Iyengar/Krause:2012b}.
\end{chunk}

\subsection*{Applications}
We record some consequences of the Theorems~\ref{th:stratification} and \ref{th:costratification}. 
The one below is by \cite[Corollary~9.9]{Benson/Iyengar/Krause:2012b}.

\begin{corollary}
\label{ico:locandcoloc}
Let $R$ and $G$ be as before. The map sending a localising subcategory $\mathsf{S}$ of $\Rep(G, R)$ to its right orthogonal
$\mathsf{S}^\perp$ induces a bijection
\[ 
\left\{\begin{gathered} \text{tensor closed
  localising}\\ \text{subcategories of $\Rep(G, R)$}
\end{gathered}\;
\right\} \longiso \left\{
\begin{gathered}
  \text{Hom closed colocalising}\\ \text{subcategories of
    $\Rep(G, R)$}
\end{gathered}
\right\}\,.
\]
The inverse map sends a colocalising subcategory $\mathsf{S}$ to its left orthogonal
$^\perp\mathsf{S}$.  \qed
\end{corollary}

One gets also (co)stratification of the big stable category of lattices.

\begin{corollary}
\label{co:stmod}
\pushQED{\qed} 
The tensor triangulated category $\StlfMod (G,R)$ is stratified and costratified by the action of $S$, and its support equals $\Prj S$. Consequently
\[
\Spc(\stmod(G,R)) \longiso \Prj H^*(G,R). 
\]
\end{corollary}

\begin{proof}
One can prove this along the lines of Theorem~\ref{th:costratification}. It can also be deduced from Corollary~\ref{co:recollement}, which implies that the canonical functor $\lfK(\Prj G,R) \to \StlfMod(G,R)$ is a finite localisation inducing the inclusion $\Prj S \to \Spec S$.
\end{proof}

Given an object $M$ in $\lfK(\Prj G,R)$ we write $\supp_GM$ and $\mathrm{cosupp}_GM$ for its support and cosupport, respectively, computed using the action of $H^*(G,R)$ on $\lfK(\Prj G,R)$; see Section~\ref{se:lch}.  Here is a noteworthy consequence of the fact that $\lfK(\Prj G,R)$ is stratified by the action of $H^*(G,R)$: For $M,N$ in $\lfK(\Prj G,R)$ one has equalities
\begin{equation}
\begin{split}
    \supp_G(M\otimes_RN)&= \supp_G M\cap \supp_G N \\
    \mathrm{cosupp}_G\fHom(M,N) &= \supp_G M\cap\mathrm{cosupp}_G N\,.
\end{split}    
\end{equation}
The formula for the support of the tensor product can be deduced from Corollary~\ref{co:stmod} as in the proof of \cite[Theorem~11.1]{Benson/Iyengar/Krause:2011b}. The cosupport of the function objects follows from \cite[Theorem~9.5]{Benson/Iyengar/Krause:2012b}. 

 Given a subgroup scheme $H\leq G$, restriction induces a map of graded $R$-algebras $H^*(G,R)\to H^*(H,R)$; let  $\mathrm{res}^*$ denote the induced map on spectra. For any $M$ in $\lfK(\Prj G,R)$ there are equalities
\begin{equation}
\begin{split}
\supp_H(M\downarrow_H) &= (\mathrm{res}^*_{G,H})^{-1}(\supp_GM) \\
\mathrm{cosupp}_H(M\downarrow_H) &= (\mathrm{res}^*_{G,H})^{-1}(\mathrm{cosupp}_GM)\,.
\end{split}
\end{equation}

There are further applications, including a proof of the telescope conjecture for $\lfK(\Prj G,R)$ and $\StlfMod(G,R)$, and versions of Theorems~\ref{th:stratification} and \ref{th:costratification} for the Frobenius category $\lfMod(G,R)$. The statements and proofs are analagous to those in \cite{Benson/Iyengar/Krause:2011a,Benson/Iyengar/Krause:2012b}, which deal with the case when $R$ is a field.

\section{An example}
\label{se:example}
In this section we give an example, promised in Section~\ref{se:Frobenius} of a group $G$ and a noetherian commutative ring $R$ with the property that the category $\lfMod(G,R)$ is properly contained in $\Mod(G,R)$.

Let $k$ be a field of characteristic $2$, set $R\colonequals k[t]/(t^2)$ and $G\colonequals \bbZ/2$, the elementary abelian $2$-group of rank one. Setting $x\colonequals g-1$, where $g$ is the generator of $G$, one gets an isomorphism 
\[
RG \cong \frac{k[t,x]}{(t^2,x^2)}\,,
\]
which identifies with the group algebra over $k$ of $E\colonequals (\bbZ/2)^2$, the Klein $4$-group.  The cohomology algebra is
\[
H\colonequals H^*(E,k)\cong k[t^*,x^*]\,,
\]
where the basis $t^*,x^*$ of $H^1$ is dual to the basis $t,x$ of $(t,x)/(t,x)^2$. One has
\[
\lfMod(G,R) \subseteq   \Mod(G,R) \subseteq \Mod(E,k)\,.
\]
This induces full and faithful inclusions
\begin{equation}
\label{eq:example-klein}
\StlfMod(G,R) \subseteq   \StMod(G,R) \subseteq \StMod(E,k)\,.
\end{equation}

Here is a key observation about these categories.

\begin{lemma}
\label{le:example-klein}
Both $\StMod(G,R)$  and $\StlfMod(G,R)$ are localising subcategories of $\StMod(E,k)$, with
\begin{align*}
\supp_H \StMod(G,R) &= \Prj H\setminus\{(t^*)\}\\
\supp_H \StlfMod(G,R) &= \Prj H \setminus\{(0), (t^*)\}\,.
\end{align*}
\end{lemma}

\begin{proof}
A $kE$-module $M$ is in $\Mod(G,R)$  if and only if  it is projective when restricted to $R$; equivalently, its support does not contain $(t^*)$. This justifies the claim about the support of $\StMod(G,R)$. 

The support of any  finite dimensional $kE$-module $M$ is closed, so $M$ is in $\rmod(G,R)$ if and only if the subset $\{(0),(t^*)\}$ is not in its support. Thus this subset is not in the support of any module that is in the localising subcategory of $\Mod(G,R)$ generated by  $\rmod(G,R)$. This justifies the second equality.
\end{proof}

It follows from the result above that the \emph{generic module} in $\StMod(E,k)$, that is to say, the $kE$-module $\gam_{(0)}k$ is in $\StMod(G,R)$ but not in $\StlfMod(G,R)$.

The inclusions in \eqref{eq:example-klein} do not respect the tensor triangular structures, for the product on the two categories on the left is $-\otimes_R-$ whereas for that on the right is $-\otimes_k-$. This is an important point, and is discussed further below. 

\subsection*{Thick and localising ideals}
We continue with a description of thick and localising ideals of the stable categories $\StlfMod(G,R)$ and $\StMod(G,R)$. 

The ideal $(t)\subset RG$ is isomorphic to $R$ as an $RG$-module, so the exact sequence 
\[
0\lra (t) \lra RG \lra R\lra 0
\]
defines an element in $H^1(G,R)=\Ext^1_{RG}(R,R)$; denote it $\eta$. Since $H^*(G,R)$ is a commutative $R$-algebra, one gets a morphism of $R$-algebras
\[
R[\eta]\lra H^*(G,R)\,.
\]
A simple computation yields:

\begin{lemma}
The map above is an isomorphism. \qed
\end{lemma}

Thus $\Prj H^*(G,R)$ consists of a single point, in contrast with $\Prj H^*(E,k)=\mathbb{P}^1_k$.  This means that there is a single nonzero thick \emph{ideal} in $\stmod(G,R)$, whereas thick subcategories are aplenty, and parameterised by specialisation closed subsets of $\mathbb{P}^1_k \setminus \{(0)\}$. Here is an example: Assume $k$ is algebraically closed and set
\[
M_{[a,b]}\colonequals RG/(at+bx)\quad \text{for $[a,b]\in \mathbb{P}^1_k$.}
\]
Observe that these modules are in $\rmod(G,R)$, except the one for $[a,b]=[1,0]$. Since there is only a single thick ideal in $\stmod(G,R)$, these modules must tensor-generate each other. Indeed as long as $[a,b]\ne [1,0]$,  as $G$-modules one has
\[
M_{[a,b]}\otimes_R M_{[-a,b]}\cong M_{[0,b]}\cong R\,.
\]

We deduce also that $\StlfMod(G,R)$ has no nontrivial localising ideals. This is another measure of the difference between this category and $\StMod(G,R)$: The modules in the latter are precisely those $kE$-modules not supported on the closed point $V\colonequals \{[1,0]\}$, so  
\[
\StMod(G,R) = L_V \StMod(E,k)\,.
\]
It follows that $\StMod(G,R)$ is also compactly generated, with compact objects generated by  $L_Vk$.  Moreover the minimal localising subcategories of $\StMod(G,R)$ are parameterised by the points of $\mathbb{P}^1_k \setminus V$. This set consists of closed points and the generic point $\{0\}$. The localising subcategories corresponding to the closed points are compactly generated and hence contained in $\StlfMod(G,R)$. Thus there are only two proper localising ideals in $\StMod(G,R)$, namely, $\StlfMod(G,R)$ and the localising subcategory generated by the $G$-module $\gam_{(0)}k$ supported at $(0)$.

A similar analysis can be made of the categories $\lfK(\Prj G, R)$ and $\bfK(\Prj G)$. Now there is one more minimal localising ideal to take into account, namely the $K$-projective complexes in $\bfK(\Prj G)$. 

\begin{chunk}
\label{ch:example-general}
The preceding construction can be generalised as follows to yield Frobenius $R$-algebras $A$ for which $\StlfMod(A,R)$ is smaller than $\StMod(A,R)$. 

Let $G$ be a commutative finite  group scheme defined over a field $k$ of positive characteristic $p$, such that the cohomology ring $H^*(G,k)$ has Krull dimension at least  $2$. Choose a  $\pi$-point 
\[
\alpha\colon R\colonequals k[t]/(t^p)\lra  kG \equalscolon A
\]
such that the corresponding point $[\alpha]$ in $\Prj H^*(G,k)$ is closed and has codimension $\ge 1$. Then $A$ is a Frobenius $R$-algebra and $\StMod(A,R)$ consists of those $G$-modules whose support does not contain $[\alpha]$. Choose a non-closed point $\fp$ in $\Prj H^*(G,k)$ whose closure contains $[\alpha]$; it exists because the codimension of $[\alpha]$ is positive. The module $\gam_{\fp}k$ is then in $\StMod(A,R)$ but not in $\StlfMod(A,R)$. 

The construction above carries over to any finite flat group scheme $G$ and $\pi$-point $\alpha$ whose image is central in $kG$.
\end{chunk}

\bibliographystyle{amsplain}

\newcommand{\noopsort}[1]{}
\providecommand{\bysame}{\leavevmode\hbox to3em{\hrulefill}\thinspace}
\providecommand{\MR}{\relax\ifhmode\unskip\space\fi MR }
\providecommand{\MRhref}[2]{%
  \href{http://www.ams.org/mathscinet-getitem?mr=#1}{#2}
}
\providecommand{\href}[2]{#2}

\end{document}